\documentclass[11pt]{amsart}
\usepackage{amssymb,amsfonts, color,epsf}
\usepackage [cmtip,arrow]{xy}
\xyoption{all}
\usepackage {pb-diagram,pb-xy}
\usepackage{enumerate}
\usepackage{accents}
\usepackage[noautoscale]{youngtab}
\usepackage{subfigure}
\ifx\pdfpageheight\undefined
   \usepackage[dvips,colorlinks=true,linkcolor=blue,citecolor=red,%
      urlcolor=green]{hyperref}
   \usepackage[dvips]{graphicx}
   \makeatletter
   \edef\Gin@extensions{\Gin@extensions,.mps}
   \makeatother
\else
 \usepackage[pdftex]{graphicx}
   \usepackage[bookmarksopen=false,pdftex=true,breaklinks=true,%
      backref=page,pagebackref=true,plainpages=false,%
      hyperindex=true,pdfstartview=FitH,colorlinks=true,%
      pdfpagelabels=true,colorlinks=true,linkcolor=blue,%
      citecolor=red,urlcolor=green,hypertexnames=false%
      ]%
    {hyperref}
\fi
\usepackage{bm}
\usepackage{tcolorbox}

\usepackage[margin=1in]{geometry}

\usepackage{amsmath,amssymb,amsfonts}
\newtheorem{theorem}{Theorem}
\newtheorem{lemma}{Lemma}

\newtheorem{proposition}{Proposition}
\newtheorem{claim}{Claim}
\newtheorem*{claim*}{Claim}
\newtheorem{conjecture}{Conjecture}

\newtheorem*{theorem*}{Theorem}
\newtheorem*{corollary*}{Corollary}
\theoremstyle{definition}
\newtheorem{definition}{Definition}

\newtheorem{notation}{Notation}

\newtheorem{problem}{Problem}
\newtheorem*{problem*}{Problem}

\usepackage{algorithm}
\usepackage{algpseudocode}

\usepackage{tikz}

\algnewcommand\algorithmicinput{\textbf{Input:}}
\algnewcommand\INPUT{\item[\algorithmicinput]}
\algnewcommand\algorithmicoutput{\textbf{Output:}}
\algnewcommand\OUTPUT{\item[\algorithmicoutput]}

\algnewcommand\algorithmicproc{\textbf{Procedure:}}
\algnewcommand\PROCEDURE{\item[\algorithmicproc]}
\algnewcommand\algorithmiccomplexity{\textbf{Complexity:}}
\algnewcommand\COMPLEXITY{\item[\algorithmiccomplexity]}

\newlength{\continueindent}
\usepackage{etoolbox}
\makeatletter
\newcommand*{\ALG@customparshape}{\parshape 2 \leftmargin \linewidth \dimexpr\ALG@tlm+\continueindent\relax \dimexpr\linewidth+\leftmargin-\ALG@tlm-\continueindent\relax}
\apptocmd{\ALG@beginblock}{\ALG@customparshape}{}{\errmessage{failed to patch}}
\makeatother

\theoremstyle{remark}
\newtheorem{remark}{Remark}

\theoremstyle{observation}

\definecolor{DarkBlue}{rgb}{0,0.1,0.55}

\numberwithin{equation}{section}

\newcommand {\hide}[1]{}
 \newcommand {\sign} {\mbox{\bf sign}}

\newcommand {\junk}[1]{}

\newcommand {\R} {\mathrm{R}}

\newcommand {\D}     {\mbox{\rm D}}

\newcommand {\C}     {\mathrm{C}}

\newcommand {\Sphere}{\mbox{${\bf S}$}}     

\newcommand {\Z}  {\mathbb{Z}}

\newcommand {\RR} {{\mathcal R}}

\newcommand {\Der} {{\rm Der}}

\newcommand {\la}   {{\langle}}
\newcommand {\ra}   {{\rangle}}
\newcommand {\eps} {{\varepsilon}}
\newcommand {\E} {{\rm ext}}

\newcommand {\B} {\mbox {\rm B}}

\newcommand{\card}{\mathrm{card}}

\def\addots{\mathinner{\mkern1mu
\raise1pt\vbox{\kern7pt\hbox{.}}
\mkern2mu\raise4pt\hbox{.}\mkern2mu
\raise7pt\hbox{.}\mkern1mu}}

\newcommand{\HH}  {\mbox{\rm H}}

\newcommand{\x}{\mathbf{x}}

\newcommand{\y}{\mathbf{y}}

\newcommand{\z}{\mathbf{z}}

\newcommand{\FF}{\mathbb{F}}

\newcommand{\Top}{\mathbf{Top}}

\newcommand{\Simp}{\mathbf{Simp}}

\newcommand{\ass}{\mathrm{ass}}

\newcommand{\clos}{\mathrm{clos}}

\newcommand{\head}{\mathrm{head}}
\newcommand{\tail}{\mathrm{tail}}

\begin{document}
\title[Efficient computation of a semi-algebraic basis]
{
Efficient computation of a semi-algebraic basis of the first homology group of a semi-algebraic set
}
\author{Saugata Basu}
\address{Department of Mathematics,
Purdue University, West Lafayette, IN 47907, U.S.A.}
\email{sbasu@math.purdue.edu}

\author{Sarah Percival}
\address{Department of Mathematics, 
Purdue University, West Lafayette, IN 47907, U.S.A.}
\email{sperciva@purdue.edu}

\date{\textbf{\today}}
\thanks{
Basu was  partially supported by NSF grants
CCF-1618918, DMS-1620271 and CCF-1910441.
}

\begin{abstract}
Let $\R$ be a real closed field and $\C$ the algebraic closure of $\R$.
We give an algorithm for computing a semi-algebraic basis for the first homology group,
$\HH_1(S,\mathbb{F})$,  with coefficients in a field $\FF$,
of any given semi-algebraic set $S \subset \R^k$ defined by a closed formula.
The complexity of the algorithm is bounded singly 
exponentially. More precisely, if the given quantifier-free formula involves $s$ polynomials whose degrees are
bounded by $d$, the complexity of the algorithm is bounded by $(s d)^{k^{O(1)}}$.
This algorithm generalizes well known algorithms having singly exponential complexity for computing a semi-algebraic 
basis of the zero-th homology group of semi-algebraic sets,  which is equivalent to
the problem of computing a set of points meeting every semi-algebraically connected component of the given semi-algebraic set at a unique point. It is not known how to compute such a basis for the higher homology groups with singly exponential complexity.

As an intermediate step in our algorithm we construct a semi-algebraic subset 
$\Gamma$ of the given semi-algebraic set $S$, such that $\HH_q(S,\Gamma) = 0$
for $q=0,1$. We relate this construction to a basic theorem in complex algebraic geometry stating that for any affine variety
$X$ of dimension $n$, there exists  Zariski closed subsets 
\[
Z^{(n-1)} \supset \cdots \supset Z^{(1)} \supset Z^{(0)} 
\]
with $\dim_\C Z^{(i)} \leq  i$, and $\HH_q(X,Z^{(i)}) = 0$ for $0 \leq q \leq i$.
We conjecture a quantitative version of this result in the semi-algebraic category, with $X$ and $Z^{(i)}$ replaced by closed semi-algebraic sets.
We make initial progress on this conjecture by proving the existence of $Z^{(0)}$ and $Z^{(1)}$ 
with complexity bounded singly exponentially (previously, 
such an algorithm was known only for constructing 
$Z_0$).
\end{abstract}

\maketitle
\tableofcontents

\section{Introduction}
We fix a real closed field $\R$, and denote by $\D \subset \R$ a fixed ordered
domain. We will denote by $\C = \R[i]$ the algebraic closure of $\R$.
For example, one can take $\R = \mathbb{R}$, and $\D = \mathbb{Z}$.

Semi-algebraic sets are subsets of $\R^k, k \geq 0$, which are defined by
quantifier-free first-order formulas with atoms of the form $P=0, P > 0, P \in \R[X_1,\ldots,X_k]$.
Algorithmic semi-algebraic geometry deals with computing geometric and
topological invariants of semi-algebraic subsets of $\R^k$, and is a very well-developed 
topic. Examples of algorithmic problems in semi-algebraic geometry that have been
investigated include effective quantifier-elimination, the decision problem of the first
order theory of the reals, computing topological invariants, such as the dimension,
number of semi-algebraically connected components, the Euler-Poincar\'e characteristic (appropriately defined), and more generally the Betti numbers of a given semi-algebraic set.

\subsection{Background and history}
\label{subsec:history}
The problem of computing the Betti numbers (i.e. the ranks of homology groups) of 
semi-algebraic sets has a long history and is an active area of current research.
Since doubly exponential complexity algorithms for computing the Betti numbers follow from effective triangulation algorithms for semi-algebraic sets, the emphasis has been on
obtaining algorithms with singly exponential complexity (see Section~\ref{subsec:single-double} below).
Singly exponential algorithms for computing the zero-th Betti number (i.e. the number of connected components) of semi-algebraic sets via construction of roadmaps
was obtained by several authors and
the complexity of the algorithms successively improved over the years 
\cite{Canny93a, GR92,  GV92, HRS93, BPR99}.

An algorithm with singly exponential complexity is known for
computing the first Betti number of semi-algebraic sets and is given in \cite{BPRbettione}, and then extended to the first
 $\ell$ (for any fixed $\ell$) Betti numbers in \cite{Bas05-first}.
The Euler-Poincar\'e characteristic, which is the alternating sum of the Betti numbers, is easier to compute,
and a singly exponential algorithm for computing it is known \cite{Basu1,BPR-euler-poincare}.
While many advances have been made in recent years \cite{Bas05-first, BPRbettione},
the best algorithm for computing \emph{all}  the Betti numbers of any given semi-algebraic set $S \subset \R^k$ 
still has doubly exponential (in $k$) complexity, even in the case where the degrees of the defining polynomials
are assumed to be bounded by a constant ($\geq 2$) \cite{SS} (here we are talking about exact algorithms, see Section~\ref{subsubsec:condition} below for a different model). 
The existence of algorithms with \emph{singly exponential complexity}  for computing all the Betti numbers of 
a given semi-algebraic set is considered to be a major open question in algorithmic semi-algebraic geometry
(see the survey \cite{Basu-survey}). \\

Unlike the singly exponential complexity algorithms for computing the zero-th
Betti numbers, the algorithms for computing the higher Betti numbers do not produce
a semi-algebraic basis. 
Obtaining such a basis efficiently and of small complexity is of interest in geometric
applications. 
In classical algebraic geometry over algebraically closed fields, representing homology classes by \emph{algebraic} cycles is a well studied problem with deep connections to 
Hodge theory. The existence of a \emph{semi-algebraic} basis with 
singly exponential complexity for the higher homology groups of a given semi-algebraic 
set is not known (other than in the zero-th homology case discussed above).
In this paper we remedy this deficiency by proving the existence of such a basis
in the case of the first homology group (cf. Theorem~\ref{thm:main}).

We note here that in the category of finite simplicial complexes, obtaining 
optimal representatives (as cycles) of homology classes is a well-studied problem. 
Early work on efficient algorithms for obtaining a shortest set of loops generating the first homology group
of a two dimensional oriented manifold (given as a simplicial complex) appears 
in \cite{Erickson05}
(see for example, \cite{DHM2020,Obayashi2018} for recent work on this topic). 
However, these algorithms are primarily combinatorial in nature and the main difficulty
in the semi-algebraic version of the problem is precisely that a triangulation of the given semi-algebraic set is not available (at least not known to be computable within
the complexity we are aiming for).

\subsubsection{Exact vs numeric}
\label{subsubsec:condition}
We remark here that  by the word ``algorithm''  in the previous paragraphs we are referring only to algorithms that work correctly for all inputs and whose complexity is uniformly bounded, i.e. bounded in terms of the degrees
and the number of input polynomials and independent of the actual coefficients of the polynomials (so in particular they always terminate). 
In contrast to this exact/symbolic model which is valid over arbitrary
real closed fields, in numerical analysis it is common to consider algorithms whose
complexity do depend on the coefficients (via a condition number). Such algorithms
work only over the field of real numbers and might not terminate on ill-conditioned
inputs (i.e. if the condition number is infinite). In this latter model, 
algorithms
with singly exponential complexity for computing all the 
Betti numbers of semi-algebraic sets have been developed \cite{BCL2019, BCT2020.1, BCT2020.2}. As noted above, these algorithms will fail to produce any result on certain inputs. Also, they do not produce semi-algebraic bases for the homology groups. 
In this paper we will be concerned only with exact algorithms that work for all possible inputs.

\subsection{Model of computation and definition of complexity}
There are several models of computation that one can consider while dealing with semi-algebraic sets (and also several notions of what constitutes an algorithm). If
the real closed field $\R =\mathbb{R}$, and $\D = \mathbb{Z}$, one can consider these
algorithmic problems in the classical Turing model and measure the bit complexity of 
the algorithms. In this paper, we will follow the book \cite{BPRbook2} and take a more general approach valid over arbitrary real closed fields. In the particular case, when
$\D = \mathbb{Z}$, our method will yield bit-complexity bounds. The precise notion of
complexity that we use is defined in Definition~\ref{def:complexity} below.

\subsubsection{Definition of complexity}
We will use the following notion of ``complexity''  in this paper. We follow the same definition as used in the book \cite{BPRbook2}. 
  
\begin{definition}[Complexity of algorithms]
\label{def:complexity}
In our algorithms we will usually take as input quantifier-free first order formulas whose terms
are  polynomials with coefficients belonging to an ordered domain $\D$ contained in a real closed field $\R$.
By \emph{complexity of an algorithm}  we will mean the number of arithmetic operations and comparisons in the domain $\D$.
If $\D = \mathbb{R}$, then
the complexity of our algorithm will agree with the  Blum-Shub-Smale notion of real number complexity \cite{BSS}.
\footnote{
In case $\D = \Z$, it is possible to deduce the bit-complexity of our algorithms in terms of the bit-sizes of the coefficients
of the input polynomials, and this will agree with the classical (Turing) notion of complexity. We do not state the bit complexity separately in our algorithms, but note that it is always bounded by a polynomial in the bit-size of the input times the complexity upper bound stated in the paper.} 
\end{definition}  
It is also useful for what follows to introduce the following mathematical definition
of ``complexity'' of formulas and semi-algebraic sets.

\subsubsection{$\mathcal{P}$-formulas, $\mathcal{P}$-semi-algebraic sets, realizations}
\begin{notation}[$\mathcal{P}$-formulas, $\mathcal{P}$-closed
formulas and their realizations]
  \label{not:sign-condition} 
  For any finite set of polynomials $\mathcal{P}
  \subset \R [ X_{1} , \ldots ,X_{k} ]$, 
  we call a quantifier-free first order formula $\Phi$ with atoms $P =0, P < 0, P>0, P \in \mathcal{P}$, to
  be a \emph{$\mathcal{P}$-formula}. 
  Given any semi-algebraic subset $Z \subset \R^k$,
  we call the \emph{realization of $\Phi$ in $Z$},
  namely the semi-algebraic set
  \begin{eqnarray*}
    \RR(\Phi,Z) & := & \{ \mathbf{x} \in Z \mid
    \Phi (\mathbf{x})\}
  \end{eqnarray*}
  a \emph{$\mathcal{P}$-semi-algebraic subset of $Z$}.
  If $Z = \R^k$, we often denote the realization of $\Phi$ in $\R^k$ by
  $\RR(\Phi)$.
  
  We say that a quantifier-free formula $\Phi$ is \emph{closed} if it is a formula in disjunctive normal form with no negations, and with atoms of the form $P \geq 0, P \leq 0$, where $P \in \R[X_1,\ldots,X_k]$. If the set of polynomials appearing in a closed formula is contained in a finite set $\mathcal{P}$, we will call such a formula a $\mathcal{P}$-closed formula,
 and we call the realization, $\RR(\Phi,\R^k)$, a \emph{$\mathcal{P}$-closed semi-algebraic set}.
  \end{notation}

\begin{definition}[Complexity of semi-algebraic sets]
For $\mathcal{P} \subset \D[X_1,\ldots,X_k]$, and 
$\mathcal{P}$-formula (resp. $\mathcal{P}$-closed formula) $\Phi$, 
we say that the complexity of $\Phi$
is bounded by $C$, where $C = \card(\mathcal{P})\cdot d$, where $d = \max_{P \in \mathcal{P}} \deg(P)$. If $S = \RR(\Phi,\R^k)$, then we will say that the complexity of $S$ is bounded by $\card(\mathcal{P})\cdot d$. 
\end{definition}

For the rest of the paper we fix a field $\FF$.

\begin{notation}
\label{not:homology}
For any closed semi-algebraic set
$X$, we will denote by $\HH_i(X) = \HH_i(X,\FF)$ 
the $i$-th homology group of $X$ with coefficients in 
$\FF$ (we refer the reader to \cite[Chapter 6]{BPRbook2} for definition of homology
groups of semi-algebraic subsets of $\R^k$, where $\R$ is an arbitrary real closed field).
\end{notation}


\begin{notation}
\label{not:balls-spheres}
Given $\x \in \R^k, r > 0$,
we will denote by $B_k(\x,r) \subset \R^k$ the (open) euclidean ball of radius $r$
centered at $\x$, and by $\Sphere^{k-1}(\x,r) \subset \R^k$, the sphere of radius $r$
centered at $\x$. Note that these are semi-algebraic subsets of $\R^k$.
\end{notation}

\subsection{Singly vs doubly exponential}
\label{subsec:single-double}
The problem of computing topological invariants (such as the number of semi-algebraically connected
components) of semi-algebraic sets in general is a hard problem (known to be
$\mathrm{PSPACE}$-hard in the Turing model).

From the point of view upper bounds on the complexity, these problems can be solved
by combinatorial means if we have in hand a triangulation of the given semi-algebraic set.
A semi-algebraic triangulation of a closed and bounded semi-algebraic set $S \subset \R^k$,
consists of a finite simplicial complex $K$, and a semi-algebraic homeomorphism 
$h: |K| \rightarrow S$ (where $|\cdot|$ denotes the \emph{geometric realization} functor).
Moreover, the homology group $\HH_*(S)$ is isomorphic to the simplicial homology groups
$\HH_*(K)$ which can be computed using standard linear algebra with complexity 
polynomial in the size of $K$.

Closed and bounded semi-algebraic sets admit semi-algebraic triangulations, and
more pertinently such triangulations can be effectively computed. However, the algorithms with the best complexity for computing such triangulations have \emph{doubly exponential} complexity (doubly exponential in $k$). More precisely, if $S \subset \R^k$ is defined by a quantifier-free formula involving $s$ polynomials of degrees at most $d$, the best algorithm for computing a semi-algebraic triangulation of $S$ is bounded by $(sd)^{2^{O(k)}}$.

It is a common belief in algorithmic semi-algebraic geometry that topological invariants satisfying
a certain bound (say singly exponential) should in fact be computable by algorithms with complexity reflecting the mathematical bound. So invariants which are bounded
singly exponential should in fact be computable by algorithms with singly exponentially bounded complexity. The intuition behind this belief is that the
natural way to compute
a topological invariant (such as the the zero-th Betti number of a semi-algebraic set) 
is often by computing a semi-algebraic representative or witnessing set (for example, a semi-algebraic
basis of $\HH_0(S)$ in the case of the zero-th homology -- see Problem~\ref{prob:H0'} below), and the cardinality
of this witnessing set is the invariant to be computed. One expects the complexity of 
the algorithm for computing the witnessing set should 
reflect the complexity of this set -- which includes the cardinality but also its 
``algebraic complexity'' as well.
The Betti numbers (ranks of homology groups)  
of semi-algebraic sets admit singly exponential upper bounds
\cite{OP,T,Milnor2}.
From this point of view one expects that there should exist algorithms for computing the Betti numbers of semi-algebraic sets with complexity bounded singly exponentially. 
Indeed, algorithms for computing the zero-th Betti number (i.e. the number of semi-algebraically connected components 
\footnote{The reason behind insisting on ``semi-algebraically'' connected instead of just connected (in the Euclidean topology) is that over an arbitrary real closed field these two notions are distinct. On the other hand if $\R = \mathbb{R}$, then being semi-algebraically connected is equivalent to being connected (see for example \cite[Theorem 5.22]{BPRbook2}).})
of semi-algebraic sets have been investigated in depth, and nearly optimal algorithms are known for this problem. 

\subsection{Computing a basis for $\HH_0(S)$}
\label{subsec:outline-H0}
As mentioned above, all algorithms for computing the zero-th Betti number (i.e. the number of semi-algebraically connected components) of a given semi-algebraic set actually solve the following more general problem.

\begin{problem}
\label{prob:H0}
Given a quantifier-free formula $\Phi$ defining a semi-algebraic subset $S\subset \R^k$, 
compute a finite subset $\Gamma \subset S$, such that for each connected component $C$
of $S$, $\card(\Gamma \cap C) = 1$ (i.e. $\Gamma$ contains a unique representative from each connected component of $S$).
\end{problem}

\begin{remark}
Note that one needs to allow points in $\Gamma$ whose coordinates are algebraic over
the ring generated by the coefficients of the polynomials appearing 
in the formula  $\Phi$. At this point we ignore the question of representation of such points (cf. Definition~\ref{def:Thom-encoding} below) other than commenting that any algorithm for solving this problem needs to address this issue.
\end{remark}

Problem~\ref{prob:H0} has been studied in depth and we now have very close to optimal algorithms for solving it. The solution is in two steps. 

\begin{enumerate}[Step 1.]
\item
\label{itemlabel:prob-H0:1}
The first and easier step is solving a weaker problem of computing a finite subset $\Gamma \subset S$ 
of ``sample points" with the property that
for each connected component $C$ of $S$, $\card(\Gamma \cap C) \geq  1$.
(The above property is equivalent to the property that the zero-th homology, $\HH_0(S,\Gamma)$, of the pair $(S,\Gamma)$ is trivial.)
There are now very efficient algorithms for solving this problem (see for example, \cite{GV,BPR95b}, \cite[Algorithm 13.3 (Sample points on a variety)]{BPRbook2}). In fact, such algorithms form the basic building block for efficient algorithms for solving the quantifier-elimination algorithms in the theory of the reals. 

\item
\label{itemlabel:prob-H0:2}
The second step is more complicated and involves solving the problem of deciding 
whether two given points in a semi-algebraic set $S$ 
belong to the same semi-algebraically connected component of $S$ efficiently.
This problem has a long history. 
The key idea is that of a roadmap of a semi-algebraic set (cf. Definition~\ref{def:roadmap} for a precise definition).

If we have a roadmap $\Gamma$ of $S$ 
which contains $\x,\y \in S$, then it is easy to decide whether $\x$ and $\y$ belong to the same semi-algebraically connected component of $\Gamma$, which tells us if they are in the same semi-algebraically connected component of $S$ itself because of the defining property of a roadmap. 
There exists singly exponential complexity algorithms for construction of roadmaps of semi-algebraic sets
\cite{Canny93a, GR92,  GV92, HRS93, BPR99},  
and hence for the  problem of deciding whether two given points belong to the same
semi-algebraically connected component of a semi-algebraic set.
Once we can decide if two points of the set of sample points $\Gamma$ belong to the same semi-algebraically connected component of $S$, we can then select exactly one point in every semi-algebraically connected component  of $S$ and thus obtain a singly-exponential algorithm for solving Problem~\ref{prob:H0}.
\end{enumerate}

\subsubsection{Interpretation in terms of a homology basis}
Note that Problem~\ref{prob:H0} can be reformulated in terms of homology
as follows. 

\begin{problem}
\label{prob:H0'}
Given a quantifier-free formula $\Phi$ defining a semi-algebraic subset $S\subset \R^k$, compute a semi-algebraic basis of $\HH_0(S)$. More precisely, compute a 
semi-algebraic subset $\Gamma = \{\x_1,\ldots, \x_N\} \subset S$, such that
the subspaces $[\x_1],\ldots,[\x_N] \subset \HH_0(S)$ are linearly independent and
span $\HH_0(S)$. Here $[\x_i]$ is the image of $\HH_0(\{\x_i\}) (\cong \FF)$ in $\HH_0(S)$
under the linear map induced by the inclusion $\{\x_i\} \hookrightarrow S$.
\end{problem}

The discussion in the beginning of this subsection yields the following theorem.

\begin{theorem}\cite{Canny93a, GR92,  GV92, HRS93, BPR99}
\label{thm:H0}
There exists an algorithm that takes as input a finite set 
\[
\mathcal{P} \subset \D[X_1,\ldots,X_k],
\]
and a $\mathcal{P}$-formula $\Phi$ whose  complexity is bounded by $C$, 
and outputs a semi-algebraic
basis, $\{\x_1,\ldots,\x_N\}$,  of $\HH_0(\RR(\Phi))$. The complexity of each $\x_i$ (as a semi-algebraic subset of $\R^k$ defined over $\D$), as well as the complexity of the algorithm, are both bounded by $C^{k^{O(1)}}$.
\end{theorem}

\section{New Results}
\label{sec:new}
We now state the new results proved in the paper.

\subsection{Generalization to the first homology group}
The main goal of this paper is to prove an analog of Theorem~\ref{thm:H0} with the
zero-th homology group replaced by the first homology group. 
The following theorem is the main result of the paper.

\begin{theorem}
\label{thm:main}
There exists an algorithm that takes as input a finite set 
\[
\mathcal{P} \subset \D[X_1,\ldots,X_k],
\]
and a $\mathcal{P}$-closed formula $\Phi$ whose  complexity is bounded by $C$, 
and outputs a finite set $\mathcal{Q} \subset \D[X_1,\ldots,X_k]$, 
as well as a finite tuple $(\Psi_j)_{j \in J}$, in which each $\Psi_j$ is a $\mathcal{Q}$-formula, such that
the realizations $\Gamma_j = \RR(\Psi_j,\R^k)$ have the following properties:
\begin{enumerate}[1.]
\item For each $j \in J$, 
$\Gamma_j \subset S$ and  $\Gamma_j$ is semi-algebraically homeomorphic to 
$\Sphere^1$ ( $ = \Sphere^1(\mathbf{0},1)$);
\item the inclusion map $\Gamma_j \hookrightarrow S$ induces an injective map
$\FF \cong \HH_1(\Gamma_j) \rightarrow \HH_1(S)$, whose image we denote by $[\Gamma_j]$;
\item
the tuple $([\Gamma_j])_{j \in J}$ forms a basis of $\HH_1(S)$.
\end{enumerate}
The complexity of each $\Gamma_j, j \in J$ (as a semi-algebraic subset of $\R^k$ defined over $\D$) is bounded by $C^{O(k^2)}$,
and the complexity of this algorithm is bounded by $C^{k^{O(1)}}$.
\end{theorem}

\subsection{Connections with the ``basic lemma'' in complex algebraic geometry}
\label{subsec:basic}
As mentioned previously,  designing an efficient algorithm 
(i.e. with singly exponential complexity) for computing semi-algebraic triangulations of 
semi-algebraic sets is one of the most important open problems in 
algorithmic semi-algebraic geometry. In the absence of such an algorithm, several
low dimensional work-arounds have been designed to compute important topological invariants without having to compute a full triangulation. For example, there are
algorithms with singly exponential complexity for computing the first $\ell$ Betti numbers (for any fixed $\ell \geq 0$) 
which do not use triangulations \cite{Bas05-first}. More recently,
in \cite{Basu-Karisani},  the authors give an algorithm with singly exponential
complexity for computing a simplicial complex which is homologically 
$\ell$-equivalent (see Definition~\ref{def:ell-equivalent} below) to a given closed
semi-algebraic set (for any fixed $\ell \geq 0$). 

Over the complex numbers, algebraic varieties cannot be triangulated (in the usual topological sense) using algebraic sets and maps. 
The following result  (see for example \cite{Nori2002} where it is called the 
``Basic Lemma'') serves as a substitute and can be considered a weak analog in complex algebraic geometry
of the property that real semi-algebraic sets can be triangulated using semi-algebraic maps.

\begin{lemma}[Basic Lemma - first form \cite{Nori2002}]
\label{lem:basic}
Let $K$ be a subfield of $\C$. Let $W$ be Zariski closed in an affine variety $X$ defined over $K$. Assume $\dim W < \dim_\C X$. Then, there is a Zariski closed  $Z$ in $X$ so that
$\dim_\C Z < \dim_\C X$ with $W_\C \subset Z$, and
\[
\HH_q(X,Z) = 0, 
\]
whenever $q \neq \dim_\C X$.
\end{lemma}

\begin{remark}
\label{rem:basic}
Note also that by applying Lemma~\ref{lem:basic} repeatedly to an affine variety
$X$ of dimension $n$, one obtains Zariski closed subsets of $X$,
\[
Z^{(n-1)} \supset \cdots Z^{(i)} \supset \cdots \supset Z^{(1)} \supset Z^{(0)} 
\]
with $\dim Z^{(i)} \leq i$, and $\HH_q(X,Z^{(i)}) = 0$ for $0 \leq q \leq  i$.
(This follows from applying the basic lemma with $W = \emptyset$, first to $X$ to obtain
$Z^{(n-1)}$, and then to $Z^{(n-1)}$ to obtain $Z^{(n-2)}$ and so on, and the homology long exact sequence of the various  triples $(X,Z^{(i)}, Z^{(i-1)})$.)
\end{remark}

One could ask for a semi-algebraic version of the basic lemma, where $W,X,Z$ are closed semi-algebraic subsets (or even Zariski closed subsets) of $\R^k$, where $\R$
is a real closed field. The proof of the basic lemma (actually of a stronger version)
given in \cite{Nori2002} is valid over real closed  fields.
However, the proof is not effective in the sense that no bound on the
complexity of $Z$ is given (in terms of the complexities of $X$ and $W$). However, since
the proof depends on iterated projections, the complexity of the construction is 
likely to be doubly exponential. 

Notice that in the semi-algebraic case (taking $W = \emptyset$), one
can take $Z = Z^{(i)} = h(|\mathrm{sk}_{i}(K)|)$, where $h:|K| \rightarrow X$ is a semi-algebraic
triangulation of $X$, 
and $\mathrm{sk}_{q}(K)$ denotes the $q$-dimensional skeleton of $K$. 
The sequence of semi-algebraic subsets $Z^{(\dim X -1 )} \supset \cdots \supset Z^{(0)}$
would then satisfy the properties of Remark~\ref{rem:basic}.
However, it is clear that
the best complexity one can obtain in this way is doubly exponential.
The existence of the sequence $Z^{(i)}$ over algebraically closed fields
(cf. Remark~\ref{rem:basic}) inspires the following the question.\\

Is it possible
to prove in the semi-algebraic case the existence of a similar sequence
satisfying the homological property in Remark~\ref{rem:basic}, but where the
the subsets $Z^{(i)}$ do not correspond to skeleta of some triangulation -- and hence
could be potentially of smaller complexity ? \\

In fact, it makes sense to ask
for the existence of the sequence $Z^{(i)}$ whose complexity is graded in terms of $i$ i.e. $Z^{(0)}$ has the smallest complexity, followed by $Z^{(1)}$, and so on.
We formulate below a quantitative conjecture which is a version of the  basic lemma in the semi-algebraic case as follows.

\begin{conjecture}
\label{conj:basic}
 Let $X \subset \R^k$ be a semi-algebraic set defined by a closed formula
 of complexity bounded by $C$, and let $\dim X = n$.
Then there exists closed semi-algebraic subsets of $X$,
\[
 Z^{(n-1)} \supset \cdots \supset Z^{(1)} \supset Z^{(0)} 
\]
with $\dim Z^{(i)} \leq  i$, and $\HH_q(X,Z^{(i)}) = 0$ for $0 \leq q \leq i$, such that
for each ${i, 0 \leq i \leq n-1}$,
the complexity of $Z^{(i)}$ is bounded by $C^{O(k^{i+1})}$.  

Moreover, there  exists an algorithm for computing closed formulas describing $Z^{(0)}, \ldots,Z^{(\ell)}, 0 \leq \ell \leq n-1$, whose complexity is bounded by $C^{k^{O(\ell)}}$.
\end{conjecture}

\begin{remark}
\label{rem:conj:basic}
Conjecture~\ref{conj:basic} is especially interesting because of the following observation.
The standard algorithms for triangulating 
semi-algebraic sets using cylindrical algebraic decomposition (see for example \cite[Chapter 5]{BPRbook2}) can be modified so that their complexities are bounded doubly exponentially only in the 
dimension of the given semi-algebraic set rather than that of the ambient space. Hence if  Conjecture~\ref{conj:basic} is true, then it will provide an alternative approach
(compared to \cite{Bas05-first, Basu-Karisani})
towards the problem of computing the the first $\ell$ Betti numbers of any given semi-algebraic 
set with singly exponential complexity for each fixed $\ell$. In this paper, we take 
an initial step towards verifying the conjecture (see Theorem~\ref{thm:basic}). We believe that the same inductive approach used in the proof of Proposition~\ref{prop:alg:surjection:correctness}
can be generalized to handle the full conjecture.
\end{remark}

We prove the following.
\begin{theorem}
\label{thm:basic}
With the same notation as in Conjecture~\ref{conj:basic}, there exists
closed semi-algebraic subsets of $X$,
\[
Z^{(1)} \supset Z^{(0)} 
\]
with $\dim Z^{(i)} \leq  i $, and $\HH_q(X,Z^{(i)}) = 0$ for $0 \leq q \leq i, i =0,1$,
such that
the complexities of $Z^{(0)}, Z^{(1)}$ are bounded by $C^{O(k^2)}$. 

Moreover, there  exists an algorithm for computing closed formulas describing $Z^{(0)}, Z^{(1)}$, whose complexity is bounded by $C^{k^{O(1)}}$.
\end{theorem}
\subsection{Comparison with roadmaps}
\label{subsec:roadmaps}
One of our intermediate constructions, namely the semi-algebraic subset $Z^{(1)}$ in
Theorem~\ref{thm:basic} (also the set $\Gamma$ constructed in Algorithm~\ref{alg:surjection}) is reminiscent of roadmaps of semi-algebraic sets
(mentioned earlier), and their construction is somewhat similar to the (classical) construction of roadmaps. \footnote{More modern algorithms such as those described in  \cite{BRSS14, BR14} use different techniques.} We describe here the key difference between our construction and the classical construction of roadmaps
in Section~\ref{subsec:outline-H1} where we give an outline of Algorithm~\ref{alg:surjection}. In this section, we recall the defining property of roadmaps of semi-algebraic sets and indicate why a roadmap is not sufficient 
for the purposes of the current paper.

The following definition is taken from \cite[Chapter 15]{BPRbook2}.
Let $S \subset \R^k$ be a semi-algebraic set. We denote
by $\pi_1:\R^k \rightarrow \R$ the projection to the $X_{1}$-coordinate 
and denote for $x \in \R$, 
$
S_{x} = S \cap \pi_1^{-1}(x)
$ 
(cf. Notation~\ref{not:proj} below).
\begin{definition}
\label{def:roadmap}
A semi-algebraic subset $\Gamma \subset S$ is called a roadmap of $S$ if it satisfies
the following properties:
\begin{enumerate}
    \item [$\mathrm{RM}_0$.] $\dim \Gamma \leq 1$;
    \item [$\mathrm{RM}_1$.] for every semi-algebraically connected component $C$ of $S$,
  $C \cap \Gamma$ is semi-algebraically connected;
    \item [$\mathrm{RM}_2$.] for every $x \in \R$ and for every semi-algebraically
  connected component $D$ of $S_{x}$, $D \cap \Gamma  \neq \emptyset$.
\end{enumerate}
\end{definition}

\begin{remark}
\label{rem:roadmap}
We note that the design of efficient algorithms for construction of roadmaps of semi-algebraic sets
has a long history \cite{Canny93a, GR92,  GV92, HRS93, BPR99}. The algorithm with the
best complexity can be found in \cite{BPR99}. The complexity of this algorithm  is bounded by
$s^{k'+1}d^{O(k^2)}$, where $s$ is the number of polynomials used to define the given set, $d$ a bound on their degrees, $k$ the dimension of the ambient space, and $k'$
is the dimension of a real variety containing the given set \cite{BPR99}. 
More recently, the dependence on $k$ in the exponent has been further improved using new methods
\cite{BRSS14, BR14}, but for the moment these new algorithms work only for 
algebraic (rather than semi-algebraic) sets.
\end{remark}

Notice that if $S$ is a closed semi-algebraic set and $\Gamma \subset S$ is a roadmap
of $S$, then the homomorphism $i_{*,1}:\HH_1(\Gamma) \rightarrow \HH_1(S)$ need not be surjective. A simple example of this phenomenon is provided by a torus
$T \subset \R^3$ as depicted in Figure~\ref{fig:torus}.
\begin{figure}
\centering
\begin{minipage}{.3\textwidth}
  \centering
  \includegraphics[width=.4\linewidth]{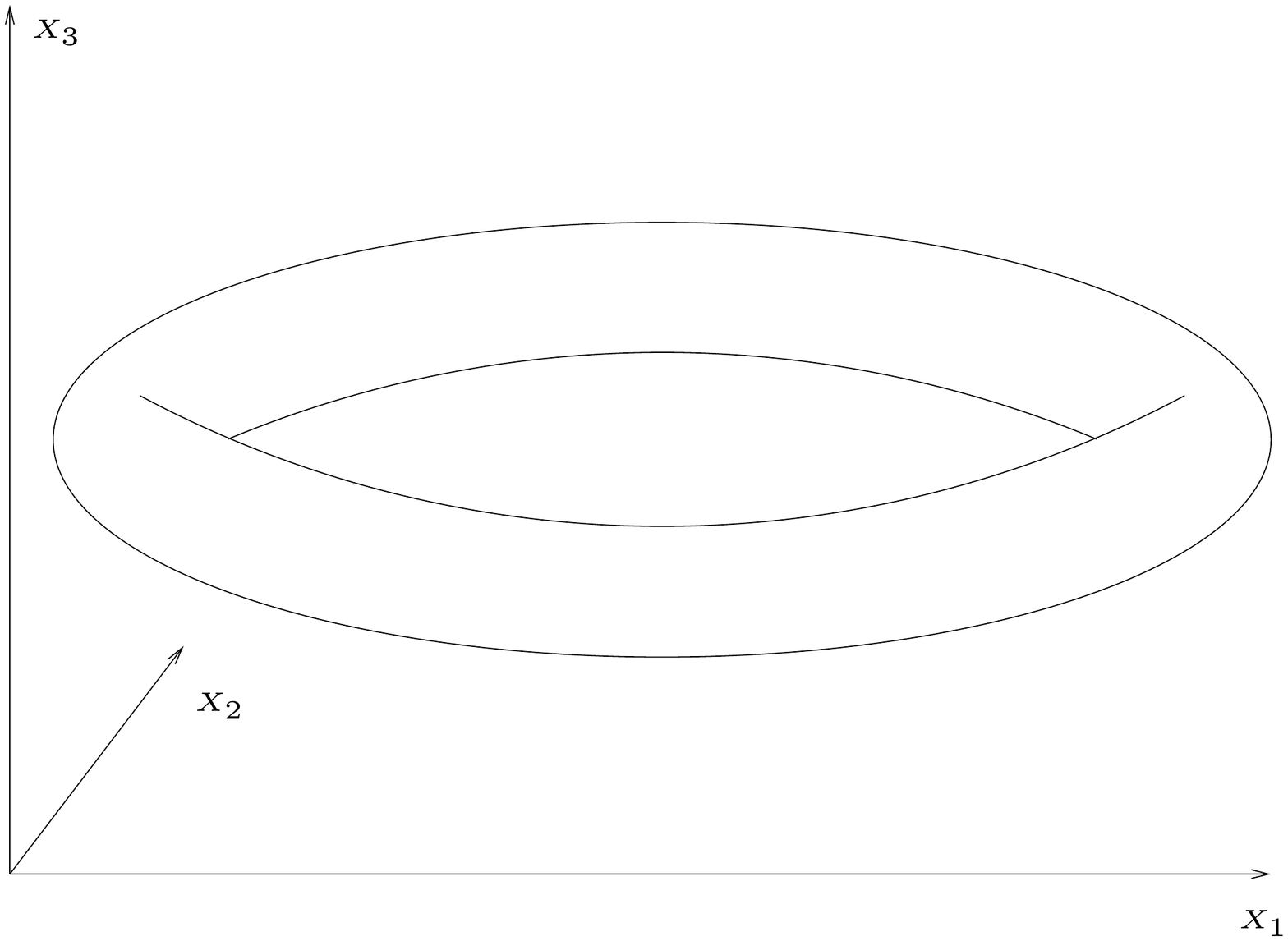}
  \caption{figure}{$T \subset \R^3$}
  \label{fig:torus}
\end{minipage}%
\begin{minipage}{.3\textwidth}
  \centering
  \includegraphics[width=.4\linewidth]{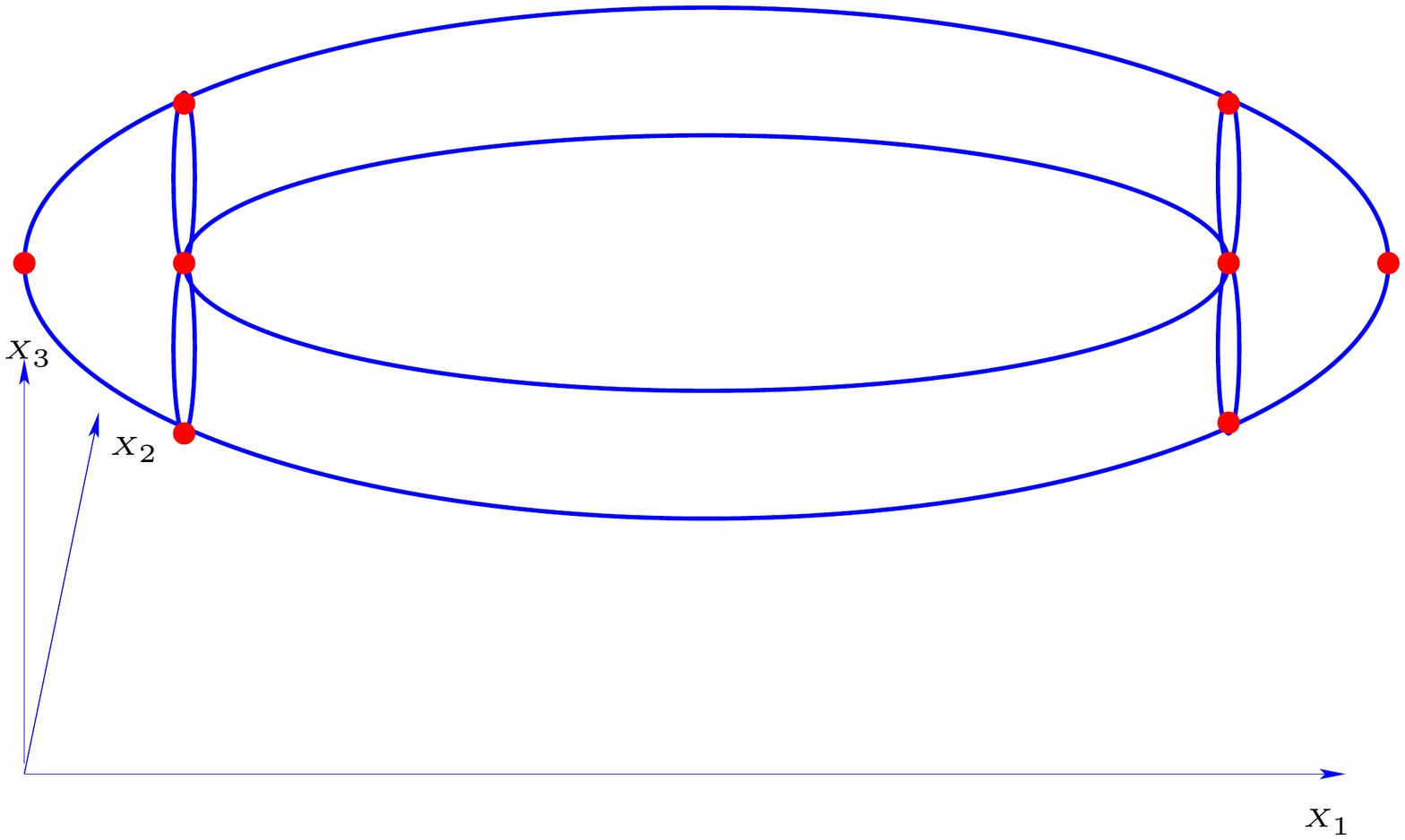}
  \caption{figure}{$\Gamma \subset T$}
  \label{fig:torus-roadmap}
\end{minipage}
\begin{minipage}{.3\textwidth}
  \centering
  \includegraphics[width=.4\linewidth]{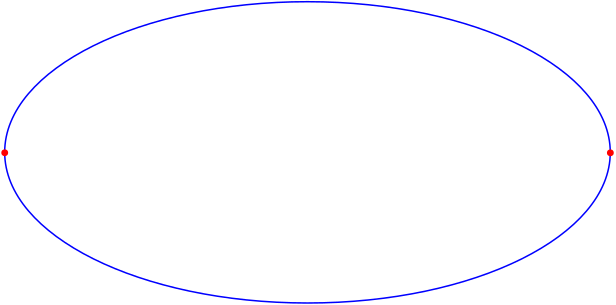}
  \caption{figure}{$\Gamma' \subset T$}
  \label{fig:torus-roadmap'}
\end{minipage}
\end{figure}
The classical construction of roadmap (see for example \cite[Chapter 15]{BPRbook2}) 
produces the one dimensional semi-algebraic subset $\Gamma$ depicted in
Figure~\ref{fig:torus-roadmap}.
In this example, the set $\Gamma$ does satisfy the property that $\HH_q(T,\Gamma) = 0, q=0,1$, though
this is not ensured by the classical roadmap algorithm.
However, notice that if we take the subset $\Gamma' \subset \Gamma$ consisting only of the larger horizontal circle as depicted in the Figure~\ref{fig:torus-roadmap'}, then
$\Gamma'$ does satisfy the property of being a roadmap of $T$ (cf. Definition~\ref{def:roadmap}), but
$\HH_1(T,\Gamma') \neq 0$ (in fact, $\dim \HH_1(T,\Gamma') = 1$). \\

The rest of the paper is devoted to proving Theorems~\ref{thm:main} and \ref{thm:basic}.
In Section~\ref{sec:prelim}, we discuss some preliminaries including definitions 
of representations of real algebraic numbers, points, and curve segments that
we use in our algorithms. In Section~\ref{sec:proof}, we prove Theorems~\ref{thm:main} and \ref{thm:basic}.

\section{Preliminaries}
\label{sec:prelim}
\subsection{Some notation}
\begin{notation}[Sign conditions]
\label{not:sign-condition}
For a finite set $\mathcal{P} \subset \R[X_1,\ldots,X_k]$ we will call any element
of $\{0,1,-1\}^{\mathcal{P}}$ a \emph{sign condition on $\mathcal{P}$}. For a semi-algebraic subset $Z \subset \R^k$,
we will denote by 
\[
\RR(\sigma) = \{\x \in Z \;\vert\; \sign(P(\x)) = \sigma(P), P \in \mathcal{P} \}.
\]

Given a sign condition $\sigma \in \{0,1,-1\}^{\mathcal{P}}$, we will denote by
$\overline{\sigma}$ the formula defined by 
\[
\overline{\sigma} = \bigwedge_{\sigma(P) = 0} (P = 0) \wedge \bigwedge_{\sigma(P) = 1} (P \geq 0) \wedge \bigwedge_{\sigma(P) = -1} (P \leq 0),
\]
and call $\overline{\sigma}$ the \emph{weak sign condition associated to $\sigma$}.
\end{notation}

\begin{notation}[Closure]
\label{not:closure}
For any semi-algebraic subset $S \subset \R^k$, we will denote by $\clos(S)$ the closure 
of $S$ (in the euclidean topology). It is a consequence of
the fact that theory of real closed fields admits 
quantifier-elimination that $\clos(S)$ is again a semi-algebraic subset of 
$\R^k$.
\end{notation}

\begin{notation}[Projections to coordinate subspaces]
\label{not:proj}
We will denote by $\pi_i:\R^k \rightarrow \R$, the projection on to the $i$-th coordinate.
More generally, for a subset $J \subset \{1,\ldots, k\}$, we denote by
$\pi_J:\R^k \rightarrow \R^{J}$, the projection on the coordinates indexed by $J$.
In particular, $\pi_{[1,i]}:\R^k \rightarrow \R^{[1,i]}$ will denote the projection onto the first $i$ coordinates.

For $1 \leq i \leq k$,  $S \subset  \R^k$, 
$J \subset \{1,\ldots,k\}$,
and $Z \subset \R^{J}$, we will denote by
$S_Z = \pi^{-1}_{J}(Z) \cap S$. If $Z = \{ \z \} \subset \R^J$, then we will write
$S_\z$ instead of $S_{\{\z\}}$.
\end{notation}

\subsection{Representations of points and curves}
While the algorithms that we describe have certain geometric underpinnings, it is important to remember that the points and
curve segments that we compute need to be represented algebraically, and hence
we need to specify the precise representations that we use. 

Moreover, we often 
fix a set of coordinates (say $(X_1,\ldots, X_i)$), to 
$\mathbf{t} = (t_1,\ldots,t_i)$ and call an algorithm recursively in the fiber
$\{\mathbf{t}\} \times \R^{k-i}$. This necessitates the introduction of triangular Thom encodings (which fixes points), and all our representations of points and curve segments
are introduced relative to such triangular Thom encodings.

The following definitions are adapted from \cite{BPRbook2}.

We begin with the representations of elements of $\R$ (which are algebraic over $\D$)
as roots of polynomials in $\D[X]$ with a given Thom encoding (cf. Definition~\ref{def:Thom-encoding} below).

\begin{definition}[Thom encoding, associated element of $\R$]
\label{def:Thom-encoding}
For $P \in \R[X]$ we will denote by 
\[
\Der(P) = \left(P,P',\ldots,P^{(\deg(P))}\right)
\]
the list of derivatives of $P$.

We will call a pair $\tau = (P, \sigma)$ with $ \sigma \in \{0,1,-1\}^{\Der (P)}$, 
the \emph{Thom encoding of $x \in \R$} 
\footnote{It is a consequence of the well-known Thom's lemma, that the Thom encoding uniquely
characterizes a root in $\R$ of a polynomial in $\D[X]$ (see for example, 
\cite[Proposition 2.27]{BPRbook2}).
},
if $\sigma (P) = 0$ and 
$\sigma(P^{(i)}) = \sign(P^{(i)}(x))$ for $0 \leq i \leq \deg(P)$.

We will sometimes abuse notation and sometime call $\sigma$ 
\emph{the Thom encoding of the root $x$ of $P$}.
We will denote $x$ by $\ass(\tau)$,  and call $\ass(\tau)$ 
\emph{the element of $\R$ associated to $\tau$}.

We will call $\deg(P)$ to be \emph{the degree of the Thom encoding $\tau$}, 
and denote it by 
$\deg(\tau)$.
\end{definition}

As remarked before we will often need to fix a block of variables
$(X_1,\ldots,X_i)$ to $\mathbf{t} = (t_1,\ldots,t_i) \in \R^i$, and perform
arithmetic operations in the ring $\D[\mathbf{t}]$. For this purpose we introduce the notion of a triangular Thom encoding whose associated point is an element of $\R^i$.

\begin{definition}[Triangular Thom encoding and associated point]
\label{def:triangular-systems}
A  \emph{triangular Thom encoding} 
$\mathcal{T} = (\mathbf{F},\boldsymbol{\sigma})$ of size $i$  is a tuple
(triangular system) of polynomials,
\[
\mathbf{F} = (f_1, \ldots, f_i)
\]
where $f_j \in \R[X_1,\ldots,X_j], 1 \leq j \leq i$,
and 
a tuple of Thom encodings $\boldsymbol{\sigma} = (\sigma_1,\ldots,\sigma_i)$,
with $\sigma_j \in \{0,1,-1\}^{\Der_{X_j}(f_j)}$,
such that for each $j, 1 \leq j \leq i$,
there exists $t_j \in \R$, such that
$t_j$ is a root of the polynomial $f_j(t_1,\ldots,t_{j-1},T_j)$
with Thom encoding $\sigma_j$. 
We call $(t_1,\ldots,t_i) \in \R^i$ the \emph{point associated to $\mathcal{T}$} 
and denote $\ass(\mathcal{T}) = (t_1,\ldots,t_i)$. 

Given a triangular Thom encoding 
\[
\mathcal{T}^+ = ((f_1,\ldots,f_{i+1}), (\sigma_1,\ldots,\sigma_{i+1})),
\]
with $\ass(\mathcal{T}^+) = (t_1,\ldots,t_{i+1})$,
we will sometimes call the pair $\tau = (f_{i+1},\sigma_{i+1})$ a 
\emph{Thom encoding over the triangular Thom encoding
$\mathcal{T} = ((f_1,\ldots,f_i), (\sigma_1,\ldots,\sigma_i))$}. 
In this case we will denote $t_{i+1}$ by $\ass(\tau)$ (generalizing
Definition~\ref{def:Thom-encoding}).

We will call $\max_{1 \leq j \leq i} \deg(F_j)$ the degree of the triangular
Thom encoding $\mathcal{T}$, and denote it by $\deg(\mathcal{T})$.

If $\tau = (f_{i+1},\sigma_{i+1})$ a 
is Thom encoding over a  triangular Thom encoding
\[
\mathcal{T} = ((f_1,\ldots,f_i), (\sigma_1,\ldots,\sigma_i)),
\]
we will call
$\deg_{T_{i+1}}(f_{i+1})$, the degree of $\tau$, and denote it by $\deg(\tau)$.

Finally, given a triangular Thom encoding $ \mathcal{T} = (\mathbf{F},\boldsymbol{\sigma})$
of size $i$, we denote by $\theta(\mathcal{T})$, the formula
\[
\bigwedge_{1 \leq j \leq i} \;\; \bigwedge_{0 \leq h \leq \deg_{X_j}(f_j)} \left(\sign(f_j^{(h)}) = \sigma_j(f_j^{(h)})\right).
\]
\end{definition}

\begin{notation}
\label{not:weak-Thom encoding}
Given a  triangular Thom encoding 
$\mathcal{T} = (\mathbf{F},\boldsymbol{\sigma})$ of size $i$ (following the same notation as in Definition~\ref{def:triangular-systems} above), we will denote
by $\overline{\boldsymbol{\sigma}}$, the closed formula obtained from $\boldsymbol{\sigma}$ by replacing each sign condition on the derivatives by the
corresponding weak inequality (i.e. replacing $f_j^{(h)} > 0$ by $f_j^{(h)} \geq 0$ and $f_j^{(h)}< 0$ by $f_j^{(h)} \leq 0$).
\end{notation}

It is a consequence of Thom's lemma \cite{BPRbook2} that:
\begin{lemma}
\label{lem:Thom}
Given a  triangular Thom encoding 
$\mathcal{T} = (\mathbf{F},\boldsymbol{\sigma})$ of size $i$, 
\[
\ass(\mathcal{T}) = \RR(\overline{\boldsymbol{\sigma}}, \R^i).
\]
\end{lemma}
\begin{proof}
Follows directly from \cite[Proposition 5.39]{BPRbook2} (Generalized Thom's Lemma).
\end{proof}

We will represent points in $\R^k$ using  real univariate representations \cite[pp. 465]{BPRbook2} defined below. As explained previously, we need to define this notion with a block of variables fixed by a triangular Thom encoding. In this case the first $i$ coordinates of the point are fixed by a triangular Thom encoding, and the real univariate
representation specifies the remaining $k-i$ coordinates.

\begin{definition}[Real univariate representations over a triangular Thom encoding and associated point]
\label{def:real-univariate-rep}
Let $\mathcal{T} = (\mathbf{F},\boldsymbol{\sigma})$ be a triangular Thom encoding of size 
$i, 0 \leq i< k$.

A \emph{real univariate representation $u$ in $\R^k$ over $\mathcal{T}$} 
is a pair $(F,\sigma)$
where $F = (f,g_0, g_{i+1},\dots,g_k) \in \R[X_1,\ldots,X_i,T]$, 
with $f(\ass(\mathcal{T}),T),g_0(\ass(\mathcal{T}),T)$ co-prime, 
and
$\sigma$ a Thom encoding of a real root of $f(\ass(\mathcal{T}),T) \in \R[T]$. 

We denote by $\ass(u) \in \R^k$ the point
\begin{equation}
   \left(\ass(\mathcal{T}), \frac{g_{i+1} (\ass(\mathcal{T}),\ass(\tau) )}{g_{0} (\ass(\mathcal{T}),\ass(\tau) )}
  , \ldots , \frac{g_{k} (\ass(\mathcal{T}),\ass(\tau ))}{g_{0} (\ass(\mathcal{T}),\ass(\tau))} \right) \in \R^{k},
\end{equation}
where $\tau = (f,\sigma)$ (notice that $\tau$ is a Thom encoding over $\mathcal{T}$), 
and call $\ass(u)$
\emph{the point associated to $u$}.

We will call the pair $(D_1,D_2)$, denoted $\deg(u)$, the degree of $u$, 
where $D_1 = \deg(\mathcal{T})$ and $D_2$ is the maximum of 
$\deg_T(f), \deg_T(g_0), \ldots, \deg_T(g_k)$.
\end{definition}

We describe semi-algebraic curve segments by real univariate representations parametrized by one of the coordinates. As before the following definition assumes a triangular Thom encoding fixing the first $i$ coordinates. 

\begin{definition}[Curve segment representation over a triangular Thom encoding]
\label{def:curve-segment}
Let $\mathcal{T} = (\mathbf{F},\boldsymbol{\sigma})$ be a triangular Thom encoding of size 
$i, 0 \leq i \leq  k-1$.
A \emph{curve segment representation $\gamma$ above $\mathcal{T}$} consists of:
\begin{enumerate}[(a)]
\item
Thom encodings,  $\tau_1 = \tau_1(\gamma), \tau_2 = \tau_2(\gamma)$, over $\mathcal{T}$,
with $\ass(\tau_1) < \ass(\tau_2)$;
\item
a pair $(u,\rho) = (u(\gamma),\rho(\gamma))$, 
where 
\[
u = (f,g_0,g_{i+2},\ldots,g_k) \in \R[X_1,\ldots,X_{i+1},T]^{k-i+1},
\]
and 
\[
\rho \in \{0,1,-1\}^{\Der_T(f)},
\]
such that for every $x_{i+1} \in (\ass(\tau_1), \ass(\tau_2))$ 
there exists a real root $t(x_{i+1})$ of $f(\ass(\mathcal{T}),x_{i+1},T)$ with Thom encoding
$\rho$ and \[
f (\ass(\mathcal{T}),x_{i+1}, t(x_{i+1})), g_{0} (\ass(\mathcal{T}),x_{i+1}, t(x_{i+1}))
\]
co-prime.
\end{enumerate}
We will call $(\tau_1,\tau_2)$ the \emph{Thom encoding of the 
interval of definition of $\gamma$}.

The  semi-algebraic function $h$ which maps 
$x_{i+1} \in (\ass(\tau_1),\ass(\tau_2))$ to the point
of $\R^{k}$ defined by
\[ h (x_{i+1}) = \left(\ass(\mathcal{T}), x_{i+1}, \frac{g_{i+2} (\ass(\mathcal{T}),x_{i+1},t(x_{i+1}))}{g_{0} (\ass(\mathcal{T}),x_{i+1},t(x_{i+1}))} , \ldots ,
   \frac{g_{k} (\ass(\mathcal{T}),x_{i+1},t(x_{i+1}))}{g_{0} (\ass(\mathcal{T}),x_{i+1},t(x_{i+1}))} \right), \]
is a continuous injective semi-algebraic function, and we will denote its image
by $\ass(\gamma)$.

We will also call $\lim_{x_{i+1} \rightarrow \ass(\tau_1)+} h(x_{i+1})$ 
(resp. $\lim_{x_{i+1} \rightarrow \ass(\tau_2)-} h(x_{i+1})$) (if it is defined)  the 
\emph{left end-point} (resp. \emph{right end-point}) of $\gamma$.

We will call the pair $(D_1,D_2)$, denoted $\deg(\gamma)$, the degree of $\gamma$,
where 
\[
D_1 = \deg(\mathcal{T}),
\]
and $D_2$ is the maximum of
\[
\deg(\tau_1), \deg(\tau_2), \deg_{X_{i+1,T}}(f), 
\deg_{X_{i+1,T}}(g_{0}), \deg_{X_{i+1,T}}(g_{i+2}), \ldots,  \deg_{X_{i+1,T}}(g_k).
\]

\end{definition}

\begin{remark}
\label{rem:def:curve-segment}
Note that in Definition~\ref{def:curve-segment}, $\ass(\gamma)$ is a 
semi-algebraic set of dimension one,
and if it is bounded over $\R$, then its left and right endpoints are
well-defined. 

Also note that if $i = k-1$, and $\tau_1,\tau_2$ are two Thom encodings over 
$\mathcal{T}$ with $\ass(\tau_1) < \ass(\tau_2)$, then if $\gamma$ is the
curve segment representation over $\mathcal{T}$ defined by,
\begin{eqnarray*}
\tau_1(\gamma) &=& \tau_1, \\
\tau_2(\gamma) &=& \tau_2, \\
u(\gamma) &=& ((T,1), (0,1)),
\end{eqnarray*}
then 
$\ass(\gamma) = \{\ass(\mathcal{T})\}  \times (\ass(\tau_1(\gamma)),\ass(\tau_2(\gamma)))$.
Thus, an open interval whose end points are given by Thom encodings can be represented
by the curve segment representation given above and we will later use this fact
without mention.
\end{remark}

\section{Proofs of Theorems~\ref{thm:main} and \ref{thm:basic}}
\label{sec:proof}
\subsection{Outline of the proofs of Theorems~\ref{thm:main} and \ref{thm:basic}}
\label{subsec:outline-H1}
As in the case of the zero-th homology in Section~\ref{subsec:outline-H0}
we solve the problem in two steps (cf. Steps~\ref{itemlabel:prob-H0:1} and
\ref{itemlabel:prob-H0:2} in the solution of Problem~\ref{prob:H0}).

\begin{enumerate}[Step 1.]
    \item 
    \label{itemlabel:prob-H1:1}
    In the first step, we develop an algorithm (see Algorithm~\ref{alg:surjection} below) that
takes as input a $\mathcal{P}$-closed formula $\Phi$, and produces as output
a description of a semi-algebraic subset $\Gamma \subset S = \RR(\Phi)$, having dimension $\leq 1$, and such that the homomorphism $i_{*,1}:\HH_1(\Gamma) \rightarrow \HH_1(S)$ induced by the inclusion $i:\Gamma \hookrightarrow S$ is surjective, and the homomorphism
$i_{*,0}:\HH_0(\Gamma) \rightarrow \HH_0(S)$
is an isomorphism. This is equivalent to $\HH_q(S,\Gamma) = 0$ for $q=0,1$.

This is the analog of Step~\ref{itemlabel:prob-H0:1} in the outline for the 
solution of Problem~\ref{prob:H0} above.
The main idea behind the construction of $\Gamma$ comes from the classical
construction of a roadmap of a semi-algebraic set $S$ that goes as follows
(see also the description in \cite[Chapter 15]{BPRbook2}. One first computes curves parametrized by the $X_1$-coordinate, such that the set of curves 
so constructed meets, for every $x \in \R$, each semi-algebraically connected component of $S_x$. 
The algorithm is then called recursively on a certain finite set of ``slices'',
namely on sets of the form $S_x$, where $x \in \R$ varies over a certain finite set of 
\emph{distinguished values}. The set of distinguished values
includes values at which the connectivity of the set $S_{\leq x}$ changes.
Algorithm~\ref{alg:surjection} follows a similar paradigm. The main \emph{new feature} is that the set of distinguished values at which the 
algorithm makes recursive calls satisfies a stronger property than in the case of 
roadmap algorithms (cf. Proposition~\ref{prop:modified-curve-segment:correctness} below). This stronger property of the distinguished values allows us to prove inductively  
(using a little homological algebra via the ``Five-lemma'' \cite{Bourbaki}), 
the \emph{surjectivity}
of the map $i_{*,1}:\HH_1(\Gamma) \rightarrow \HH_1(S)$. The proof of the \emph{isomorphism}
$i_{*,0}:\HH_0(\Gamma) \rightarrow \HH_0(S)$ is the same as in the classical
construction of a roadmap of a semi-algebraic set.
It follows from the exact homology sequence of the pair $(S,\Gamma)$, that
$\HH_q(S,\Gamma) = 0$, for $0 \leq q \leq 1$. The semi-algebraic set 
$\Gamma$ is initially described as a union of points and curve segments 
(cf. Definitions~\ref{def:real-univariate-rep} and \ref{def:curve-segment})).
Converting this description into an equivalent \emph{closed} formula of complexity bounded singly exponentially 
(cf. Algorithm~\ref{alg:convert-to-closed} (Conversion of curve segment representations to closed formulas)) yields a proof of Theorem~\ref{thm:basic}.

    \item
    \label{itemlabel:prob-H1:2}
    In this step we use the fact that the semi-algebraic set
$\Gamma$ is semi-algebraically homeomorphic to the geometric realization $|G|$ of a finite graph $G$ having singly exponential size, and it is a relatively easy combinatorial task to choose a
basis of simple cycles, $\Gamma_1,\ldots, \Gamma_N$, for the cycle space of $G$. The images $[\Gamma_1],\ldots, [\Gamma_N]$ (here $[\Gamma_i]$ denotes the image of $\HH_1(|\Gamma_i|)$ in $\HH_1(S)$ under the homomorphism induced by the inclusion
$|\Gamma_i| \hookrightarrow S$), span $\HH_1(S)$ but are not necessarily linearly independent. We need to select a minimal spanning subset from amongst the 
$[\Gamma_1],\ldots,[\Gamma_N]$. For this purpose we use an algorithm for replacing
a given semi-algebraic set and a tuple of subsets by a simplicial complex and a tuple of corresponding subcomplexes, which are homologically $\ell$-equivalent
(cf. Definition~\ref{def:ell-equivalent})
for any fixed $\ell$, and which has singly exponentially bounded complexity \cite{Basu-Karisani} (cf. Algorithm~\ref{alg:simplicial-replacement} below). 
We use this algorithm in the case $\ell = 1 $. This is analogous to the  usage 
of a roadmap algorithm for overcoming the corresponding obstacle in the case of 
the zero-th homology (cf. Step~\ref{itemlabel:prob-H0:2} in the outline for the 
solution of Problem~\ref{prob:H0}).
\end{enumerate}

We now describe in detail the two steps outlined  above.

\subsection{Implementing Step~\ref{itemlabel:prob-H1:1}: computing surjection}
\label{subsec:surjection}
In this section we describe an algorithm which will accomplish Step~\ref{itemlabel:prob-H1:1} of the two-step algorithm sketched out in
Section~\ref{subsec:outline-H1}. 
Recall that the goal of this step is to obtain an algorithm with singly exponential 
complexity that
computes a description of a semi-algebraic subset
$\Gamma$ having dimension at most one,  
such that $\HH_q(S,\Gamma) = 0, q =0,1$.

As mentioned previously,
we will follow the same approach of constructing a roadmap of $S$ 
(see for instance \cite[Chapters 15 and 16]{BPRbook2}), 
however with one key additional property. 

\subsubsection{Morse-type partition}
In the (classical) construction of the roadmap, one makes recursive calls to the 
roadmap constructing algorithm at certain special fibers where the $X_1$ coordinate is fixed to certain values (called ``distinguished values'' in \cite{BPRbook2}). Thee distinguished values include the values $c$ of the $X_1$ coordinate where the connectivity
of the fiber $S_c$ can change. This is sufficient for construction of roadmaps, since
the main property of the roadmap is related to connectivity (the intersection of the roadmap with any semi-algebraically connected component of $S$ should be semi-algebraically connected). 
Since in this paper we are mainly concerned with the first homology group, we need our set
of distinguished values to satisfy a more stringent property. Fortunately, there exists a
singly exponential complexity algorithm for this purpose which we are going to utilize.

The following algorithm (without a triangular Thom encoding in the input)  appears in \cite[Algorithm 6]{Basu-Karisani}. However, as remarked before in our applications, it will be necessary to fix a block of $i$ variables by a triangular Thom encoding $\mathcal{T}$, and perform the
computations over the ring $\D[\ass(\mathcal{T})]$. Each arithmetic operation in 
the ring $\D[\ass(\mathcal{T})]$ costs $D^{O(i)}$ arithmetic operations in $\D$, where 
$D = \deg(\mathcal{T})$.

\begin{algorithm}[H]
\caption{(Morse partition)}
\label{alg:filtration}
\begin{algorithmic}[1]

\INPUT
\Statex{
\begin{enumerate}[(a)]
\item
$r = \frac{a}{b}, a,b \in \D, a,b >0$;
\item A triangular Thom encoding $\mathcal{T}$ of size $i, 0 \leq i \leq k$;
\item
A finite set $\mathcal{P} = \{P_1,\ldots,P_s \} \subset \D[X_1,\ldots,X_k]$;
\item
A $\mathcal{P}$-closed formula $\Phi$.
\end{enumerate}
}

\OUTPUT
 \Statex{
 An ordered tuple $\mathcal{F} = (\tau_1,\ldots,\tau_N)$
 of Thom encodings over $\mathcal{T}$,
 with associated points $t_1 < \cdots < t_N$, with $-r \leq t_1, t_N \leq r$, and 
 such that 
 for each $j, 1 \leq j \leq N-1$, and all $t \in [t_j, t_{j+1})$ the inclusion maps 
 \[
 S_{\{\ass(\mathcal{T})\} \times (-\infty, t_j]}\hookrightarrow S_{\{\ass(\mathcal{T})\} \times (\infty, t]}
 \]
 (resp.
 $ S_{\{\ass(\mathcal{T})\} \times [t_{j+1}, \infty)}\hookrightarrow S_{\{\ass(\mathcal{T})\} \times [t, \infty)}$) 
 induce isomorphisms 
 \[
 \HH_*(S_{\{\ass(\mathcal{T})\} \times (\infty, t_j]}) \rightarrow
 \HH_*(S_{\{\ass(\mathcal{T})\} \times (\infty, t]})
 \]
 (resp.
 $
 \HH_*(S_{\{\ass(\mathcal{T})\} \times [t_{j+1}, \infty)}) \rightarrow \HH_*(S_{\{\ass(\mathcal{T})\} \times [t, \infty)})
 $
 ),
 where 
 \[
 S = \RR(\Phi,\R^k) \cap \clos(B_k(\mathbf{0},r)).
 \]
  }

\PROCEDURE
\State{Call \cite[Algorithm 6]{Basu-Karisani}  (up to Step 13
and doing all computations in the ring
$\D[\ass(\mathcal{T}$)  twice with inputs,
$(-,r,\mathcal{P}, \Phi, X_{i+1})$ and $(-,r,\mathcal{P}, \Phi, -X_{i+1})$,
and let $\mathcal{F}_1, \mathcal{F}_2$ be the set of Thom encodings over 
$\mathcal{T}$ 
output (Part (a) of the output of \cite[Algorithm 6]{Basu-Karisani}).
}

\State{Using Algorithm  12.21 (Triangular Comparison of Roots) in \cite{BPRbook2} 
order the Thom encodings in 
$\mathcal{F}_1 \cup \mathcal{F}_2$,  and merge the two sets into one ordered tuple 
$
\mathcal{F} = (\tau_1,\ldots,\tau_N)
$
such that
\[
\ass(\tau_1) < \cdots < \ass(\tau_N).
\]
}
\State{Output $\mathcal{F}$.}

\COMPLEXITY
 The complexity of the algorithm is bounded by 
 \[
 D^{O(i)} (s d)^{O(k)},
 \]
 where $s = \card(\mathcal{P})$,
 $d = \max_{P \in \mathcal{P}} \deg(P)$,
 and $D = \deg(\mathcal{T})$.
 
 Moreover,  $\deg(\tau_i) \leq d^{O(k)}$ for $1 \leq i \leq N$, and 
 the size of $\mathcal{F}$ is bounded by $(s d)^{O(k)}$.
\end{algorithmic}
\end{algorithm}

We will need the following extra property of the output of Algorithm~\ref{alg:filtration}.
\begin{proposition}
\label{prop:alg:filtration}
For each $j, 1 \leq j \leq N-1$, and all $t \in [t_j, t_{j+1}]$ the inclusion maps 
 \[
 S_{\{\ass(\mathcal{T})\} \times \{t\}}\hookrightarrow S_{\{\ass(\mathcal{T})\} \times [t_j t_{j+1}]}
 \]
 induce isomorphisms 
 \[
 \HH_*(S_{\{\ass(\mathcal{T})\} \times \{t\}}) \rightarrow
 \HH_*(S_{\{\ass(\mathcal{T})\} \times [t_j t_{j+1}]}),
 \]
 where 
 \[
 S = \RR(\Phi,\R^k) \cap \clos(B_k(\mathbf{0},r)).
 \]
\end{proposition}
\begin{proof}
Let 
\begin{eqnarray*}
A_1 &=& S_{\{\ass(\mathcal{T})\} \times (-\infty,t]},\\
A_2 &=& S_{\{\ass(\mathcal{T})\} \times [t, \infty)},\\
B_1 &=& S_{\{\ass(\mathcal{T})\} \times (-\infty,t_{j+1}]},\\
B_2 &=& S_{\{\ass(\mathcal{T})\} \times [t_j, \infty)}.
\end{eqnarray*}
Then, $A_h \subset B_h, h=1,2$ and
\begin{eqnarray*}
A_1 \cap A_2 &=& S_{\{\ass(\mathcal{T})\} \times \{t\}}, \\
B_1 \cap B_2 &=& S_{\{\ass(\mathcal{T})\} \times [t_j t_{j+1}]}, \\
\end{eqnarray*}
and 
\[
A_1 \cup A_2 = B_1 \cup B_2 = S_{\ass(\mathcal{T})}.
\]
Moreover the properties of the output of Algorithm~\ref{alg:filtration} imply that
that the homomorphisms $\HH_*(A_h) \rightarrow \HH_*(B_h), h = 1,2 $ induced by inclusions are isomorphisms. The Mayer-Vietoris exact sequence in homology (see for example, \cite[Theorem 6.35]{BPRbook2}) then yields the
following commutative diagram with exact rows and vertical homomorphisms induced by inclusion (where 
$A_{12}$ (resp. $B_{12}$) denotes $A_1\cap A_2 $ (resp. $B_1 \cap B_2$), and
$A^{12}$ (resp. $B^{12}$) denotes $A_1\cup A_2 $ (resp. $B_1 \cup B_2$)):
\[
\xymatrix{
\HH_{m+1}(A_1) \oplus \HH_{m+1}(A_2) \ar[r]\ar[d] & \HH_{m+1}(A^{12}) \ar[r]\ar[d] & \HH_m(A_{12}) \ar[d]\ar[r]&  \HH_m(A_1) \oplus \HH_m(A_2) \ar[r]\ar[d]& 
\HH_m(A^{12})\ar[d] \\
\HH_{m+1}(B_1) \oplus \HH_{m+1}(B_2) \ar[r] & \HH_{m+1}(B^{12}) \ar[r] & \HH_m(B_{12})\ar[r]&  \HH_m(B_1) \oplus \HH_m(B_2) \ar[r]& 
\HH_m(B^{12})
}.
\]
It follows from the fact that homomorphisms $\HH_*(A_h) \rightarrow \HH_*(B_h), h = 1,2 $ induced by inclusions are isomorphisms, 
and the fact that $A^{12} = B^{12}$,
that all the vertical arrows other than the 
middle one are isomorphisms. Hence, by the ``Five-lemma''
(see for example \cite{Bourbaki})
the middle arrow is also an isomorphism, thus proving the proposition.
\end{proof}

We will also need the 
following algorithm  that  takes as input a quantifier-free formula $\Phi$ and outputs 
$r \in \D, r >0$, such that $\RR(\Phi,\R^k)$ is homologically equivalent to
$\RR(\Phi,\R^k) \cap \clos(\B(\mathbf{0},r))$.

We need a definition.

\begin{notation}
\label{not:puiseux}
  For $\R$ a real closed field we denote by $\R \left\langle \eps
  \right\rangle$ the real closed field of algebraic Puiseux series in $\eps$
  with coefficients in $\R$. As a real closed field $\R\la\eps\ra$ is uniquely ordered, and this order extends the order on $\R$. It is the unique order in which $\eps > 0$
  and $\eps < x$ for every $x \in \R, x >0$. In particular, the subring $\D[\eps] \subset \R\la\eps\ra$ is ordered by: 
  $\sum_{i \geq 0 } a_i \eps^i > 0$ if and only if $a_p > 0$ where $p = \min \{i \;\vert \; a_i \neq 0 \}$.
\end{notation}

\begin{notation}
\label{not:extension}
  If $\R'$ is a real closed extension of a real closed field $\R$, and $S
  \subset \R^{k}$ is a semi-algebraic set defined by a first-order formula
  with coefficients in $\R$, then we will denote by $\E(S, \R') \subset \R'^{k}$ the semi-algebraic subset of $\R'^{k}$ defined by
  the same formula.
 \footnote{Not to be confused with the homological functor $\mathrm{Ext}(\cdot,\cdot)$.}
 It is well-known that $\E(S, \R')$ does
  not depend on the choice of the formula defining $S$ {\cite{BPRbook2}}.
\end{notation}

\begin{notation} 
\label{not:cauchy}
For $P =  a_p T^p + \cdots + a_q T^q, p \geq q \in \D[T]$,  $a_p a_q \neq 0$, we denote
\[
c'(P) = \left((p+1) \cdot \sum_i \frac{a_i^2}{a_q^2}\right)^{-1}.
\]
\end{notation}

We will use the following lemma.
\begin{lemma}
\label{lem:cauchy}
With the same notation as in Notation~\ref{not:cauchy}, if $x \neq 0, x\in \R$ is a root of $P$, the $|x| > c'(P)$. 
\end{lemma}
\begin{proof}
See Lemma 10.7 in \cite{BPRbook2}.
\end{proof}

\begin{algorithm}[H]
\caption{(Big enough radius)}
\label{alg:big-enough-radius}
\begin{algorithmic}[1]
\INPUT
\Statex{
\begin{enumerate}[(a)]
\item
A triangular Thom encoding 
$\mathcal{T} = (\mathbf{F},\boldsymbol{\sigma})$ of size $i, 0 \leq i \leq k$;
\item
a finite set $\mathcal{P} \subset \D[X_1,\ldots,X_k]$;
\item
a $\mathcal{P}$-closed formula $\Phi$ such that $\RR(\Phi,\R^k)$ is bounded.

\end{enumerate}
}
\OUTPUT
\Statex{
Elements $a,b \in \D[\ass(\mathcal{T})], a,b > 0$, such that the inclusion
map 
\[\RR(\Phi,\R^k)_{\ass(\mathcal{T})} \cap \clos(\{\ass(\mathcal{T})\} \times B_{k-i}(\mathbf{0},r)) \hookrightarrow \RR(\Phi,\R^k)_{\ass(\mathcal{T})},
\]
where $r = \frac{a}{b}$,
induces an isomorphism 
\[
\HH_*\left((\RR(\Phi,\R^k) \cap \clos(\{\ass(\mathcal{T})\} \times B_{k-i}(\mathbf{0},r)))_{\ass(\mathcal{T})}\right) \rightarrow \HH_*\left(\RR(\Phi,\R^k)_{\ass(\mathcal{T})}\right).
\]
}

\PROCEDURE
\State{$P_1 \gets Y - (X_{i+1}^2 + \cdots + X_k^2)$.}
\State{$P_2 \gets  (\eps^2(X_{i+1}^2 + \cdots + X_k^2) - 1)$.}
\State{$\widetilde{\Phi} \gets \Phi \wedge (P_1 = 0) \wedge (P_2 \leq 0) $.}
\State{$\widetilde{\D} \gets \D[\eps]$ where $D[\eps] \subset \R\la\eps\ra$ (cf. Notation~\ref{not:puiseux}).}
\State{Call Algorithm~\ref{alg:filtration} (Morse partition)  treating $Y$ as the $(i+1)$-st coordinate,
with computations occurring in the domain $\widetilde{\D}$,
and with input $\mathcal{T},\mathcal{P} \cup \{P_1,P_2\}, \widetilde{\Phi}$.}
\State{$\mathcal{Q} \in \D[\ass(\mathcal{T})][\eps] \gets$ the set of polynomials whose signs are determined during the call to  Algorithm~\ref{alg:filtration} (Morse partition) in the previous step.}
\label{line:alg:big-enough-radius:0}
\State{$c = \frac{b}{a} \gets \min_{Q \in \mathcal{Q}} c'(Q), a, b \in \D[\ass(\mathcal{T})]$ (cf. Notation~\ref{not:cauchy}).}
\label{line:alg:big-enough-radius:1}
\State{$r \gets \frac{a}{b}$.}
\State{Output $a,b$.}

\COMPLEXITY
The complexity of the algorithm is bounded by 
 \[
 D^{O(i)} (s d)^{O(k)},
 \]
 where $s = \card(\mathcal{P})$,
 $d = \max_{P \in \mathcal{P}} \deg(P)$,
 and $D = \deg(\mathcal{T})$.
\end{algorithmic}
\end{algorithm} 

\begin{proof}[Proof of Correctness of Algorithm~\ref{alg:big-enough-radius}]
Note that the formula $\widetilde{\Phi}$ defines a semi-algebraic subset 
\[
\widetilde{S} \subset \R\la\eps\ra^{i}  \times \clos(B_{k-i}(\mathbf{0},\frac{1}{\eps})).
\]
It follows from the conic structure theorem at infinity of semi-algebraic sets
(see for example \cite[Proposition 5.49]{BPRbook2})
and the Tarski-Seidenberg transfer principle (see for example \cite[Theorem 2.80]{BPRbook2}) that $\E(S,\R\la\eps\ra)_{\ass(\mathcal{T})}$ is semi-algebraically homeomorphic to
$\widetilde{S}_{\ass(\mathcal{T})}$, and hence 
\begin{equation}
\label{eqn:proof:alg:big-enough-radius:1}
    \HH_*\left(\E(S,\R\la\eps\ra)_{\ass(\mathcal{T})}\right) \cong \HH_*\left(\widetilde{S}_{\ass(\mathcal{T})}\right).
\end{equation}

It follows from the way $r$ is computed in the algorithm in 
Line~\ref{line:alg:big-enough-radius:1} and Lemma~\ref{lem:cauchy},  that $c = \frac{1}{r}$ is strictly positive, and
smaller than all strictly positive roots in $\R$ of the polynomials in $\mathcal{Q}$ (defined in 
Line~\ref{line:alg:big-enough-radius:0} of the algorithm).
It now follows from the ordering of the ring
$\D[\ass(\mathcal{T})][\eps]$ (cf.  Notation~\ref{not:puiseux})
that 
all the branchings in the algorithm, each of which depend on the determination of
the sign of an element in $\D[\ass(\mathcal{T})][\eps]$, remain the same
if $c$ is substituted for $\eps$.

It now follows from the correctness of Algorithm~\ref{alg:filtration} (Morse partition),
that the inclusion \[
\widetilde{S}_{\{\ass(\mathcal{T})\} \times (-\infty,r]} \hookrightarrow \widetilde{S}_{\{\ass(\mathcal{T})\} \times (-\infty,\frac{1}{\eps}]} = 
\widetilde{S}_{\ass(\mathcal{T})}
\]
(with $r$ as computed in the algorithm)
induces an isomorphism
\begin{equation}
\label{eqn:proof:alg:big-enough-radius:2}
 \HH_*\left(\widetilde{S}_{\{\ass(\mathcal{T})\} \times (-\infty,r]}\right)   \rightarrow
 \HH_*\left(\widetilde{S}_{\ass(\mathcal{T})}\right).
\end{equation}

Moreover, for any $r' > 0$,
$\widetilde{S}_{\{\ass(\mathcal{T})\} \times (-\infty,r']}$ is semi-algebraically homeomorphic to $\E(S,\R\la\eps\ra) \cap \{\ass(\mathcal{T})\} \times \clos(B_{k-i}(\mathbf{0},r'))$, and hence
\begin{equation}
\label{eqn:proof:alg:big-enough-radius:3}
    \HH_*\left(\widetilde{S}_{\{\ass(\mathcal{T})\} \times (-\infty,r]}\right) \cong \HH_*\left(\E(S,\R\la\eps\ra) \cap \{\ass(\mathcal{T})\} \times \clos(B_{k-i}(\mathbf{0},r))\right).
\end{equation}
Finally, for any closed semi-algebraic set $X \subset \R^k$, 
\begin{equation}
\label{eqn:proof:alg:big-enough-radius:4}
\HH_*(X) \cong \HH_*(\E(X,\R\la\eps\ra)).
\end{equation}
The isomorphisms \eqref{eqn:proof:alg:big-enough-radius:1}, \eqref{eqn:proof:alg:big-enough-radius:2}, \eqref{eqn:proof:alg:big-enough-radius:3}, and
\eqref{eqn:proof:alg:big-enough-radius:4}
imply that
\[
\HH_*\left(S_{\ass(\mathcal{T})}\right) \cong 
\HH_*(S\cap \{\ass(\mathcal{T})\} \times \clos(B_{k-i}(\mathbf{0},r))).
\]
This proves the correctness of Algorithm~\ref{alg:big-enough-radius}.
\end{proof}
\begin{proof}[Complexity of Algorithm~\ref{alg:big-enough-radius}]
The stated complexity follows from the complexity of Algorithm~\ref{alg:filtration} (Morse partition).
\end{proof}

\subsubsection{Constructing curve segments}
\label{subsec:curve-segments}
Another basic building block of our algorithm is an algorithm that takes as input
a closed formula defining a semi-algebraic set, and produces as output
a set of Thom encodings with associated elements 
$t_1 < \cdots < t_N$, and over each open interval 
$(t_j,t_{j+1})$, a set of semi-algebraic curves parametrized by the $X_1$ coordinate, 
with $t_{j} < X_1 < t_{j+1}$, described by  curve segment representations, and
for each $j$, a set of points (described by real univariate presentations) whose
first coordinate equals $t_j$, such that the set of left and right end points of the 
curves are contained in the set of points. 
The main properties that these points and curve segments need to satisfy are listed in Proposition~\ref{prop:modified-curve-segment:correctness}, stated after we 
describe the algorithm.
Two remarks are in order.

First, note that this construction of curve segments parametrized by one of the 
coordinates (in our case $X_1$), satisfying properties very similar to those listed
in Proposition~\ref{prop:modified-curve-segment:correctness}, is a first step in the
classical construction of roadmaps of semi-algebraic sets 
\cite{Canny93a, GR92,  GV92, HRS93, BPR99}.
One important extra property that
we require is that the partition of the $X_1$-axis satisfies the  property of the output of Algorithm~\ref{alg:filtration} (Morse partition),
which is a little stronger than what is needed for such a partition
in the classical roadmap constructions.  
We also ensure that distinct curves that we construct do not
intersect (cf. Proposition~ \ref{prop:modified-curve-segment:correctness}, Part~\eqref{itemlabel:prop:modified-curve-segment:correctness:4}). In the classical roadmap construction, this extra care is unnecessary.

A second remark  is that
for reasons explained previously, in the description of the algorithm
we will assume that the first $i$ coordinates
are fixed (by a triangular Thom encoding), and the ``first coordinate'' in the description given in the beginning of this subsection should be 
replaced by the ``$(i+1)$-st coordinate''. This is in fact no different
from the case of classical algorithms for computing roadmaps, which also relies on
recursive calls in which a first block of $i$ coordinates are fixed in each nested recursive with nesting depth $i$.

Finally, we note that since the main steps of the following algorithm
are quite similar to the corresponding steps
in several classical algorithms for constructing roadmaps, we omit a few details, 
giving pointers to algorithms in \cite{BPRbook2} to be used for implementing
them in an efficient way.

\begin{algorithm}[H]
\caption{(Curve segments)}
\label{alg:modified-curve-segment}
\begin{algorithmic}[1]
\INPUT
\Statex{
\begin{enumerate}[1.]
\item
A triangular Thom encoding 
$\mathcal{T} = (\mathbf{F},\boldsymbol{\sigma})$ of size $i$ with $0 \leq i \leq k-1$;
\item
a finite set $\mathcal{P} \subset \D[X_1,\ldots,X_k]$;
\item
a $\mathcal{P}$-closed formula $\Phi$ such that $\RR(\Phi,\R^k)$ is bounded.

\end{enumerate}
}

\OUTPUT
 \Statex{
\begin{enumerate}[1.]
\item A finite tuple  $\mathcal{F} =(\tau_1,\ldots,\tau_N)$ of Thom encodings over $\mathcal{T}$, with
    \[
    t_1 = \ass(\tau_1) < \cdots < t_N = \ass(\tau_N);
    \]
    \item for each $j, 1 \leq j \leq N-1$, an indexing set $I_j$, 
    and a finite tuple $\mathcal{C}_j = (\gamma_h)_{h \in I_j}$ of curve segment representations over $\mathcal{T}$,  such that for each $h \in I_j$ 
    \[
    \tau_1(\gamma_h) = \tau_j, \tau_2(\gamma_h) = \tau_{j+1}
    \]
    (we will let $C_{0} = \mathcal{C}_{N+1} = \emptyset$);
    \item for each $j, 1 \leq j \leq N$
a finite set $\mathcal{U}_j$ of real univariate representations over $\mathcal{T}$, such that for each $u \in U_j$, 
 the set of points $\{\ass(u) \;\vert \; u \in \mathcal{U}_j \}$ includes  the set of end-points of the curve segment representations  in $C_{j-1} \cup C_{j}$;
    \item mappings $L_j,R_{j-1}: I_j \rightarrow \mathcal{U}_j$, such that
    $\ass(L_j(h))$ is the left end-point of $\gamma_h$, and $\ass(R_j(\gamma_h))$ is the right end-point of $\gamma_h$.
\end{enumerate}
}

\PROCEDURE
\State{Call Algorithm~\ref{alg:big-enough-radius} (Big enough radius) with input $(\mathcal{T},\mathcal{P},\Phi)$ and compute $a,b \in \D[\ass(\mathcal{T})]$.}

\State{$\mathcal{P} \gets \mathcal{P} \cup \{b^2(X_{i+1}^{2} + \cdots +X_{k}^{2}) - a^2\}$.}
\State{$\Phi \gets \Phi \wedge b^2(X_{i+1}^{2} + \cdots +X_{k}^{2}) - a^2 \leq 0)$.}

\State{Call Algorithm~\ref{alg:filtration} (Morse partition) with $(\mathcal{T},r,\mathcal{P},\Phi)$ as input and obtain a finite set $\mathcal{F}$ of Thom encodings over $\mathcal{T}$ as output.} 
\label{line:alg:surjection:filtration}
\State{Call Algorithm 16.12 (Bounded Roadmap) in \cite{BPRbook2} with input $\mathcal{P}, \Phi$ and the radius $r$,  performing all computations over the ring $\D[\ass(\mathcal{T})]$.}
\label{line:alg:modified-curve-segment:roadmap}
\State{Retain from the output of the previous step, 
the set $\mathcal{C}$ of curve segment representations
over $\mathcal{T}$ parametrized by $X_{i+1}$, and 
the set $\mathcal{U}$ of  real univariate representations over $\mathcal{T}$.
}

\State{For each pair $\gamma, \gamma' \in \mathcal{C}$, compute a description of 
$\ass(\gamma) \cap \ass(\gamma')$ using Algorithm 14.6 (Parametrized Sign Determination) in \cite{BPRbook2}, 
and refine the set $\mathcal{C}$ to have the property that $\ass(\gamma) \cap \ass(\gamma') = \emptyset$, for all $\gamma,\gamma' \in \mathcal{C}, \; \gamma \neq \gamma'$.}
\label{line:alg:surjection:intersection}

\State{Augment the set $\mathcal{U}$ to also contain the set of real univariate representations over $\mathcal{T}$ whose associated points are the end points
of the curve segments in $\mathcal{C}$.
}

\State{Compute Thom encodings over $\mathcal{T}$ whose associated values are the $X_{i+1}$-coordinates of the asssociated points of $\mathcal{U}$, and add these to $\mathcal{F}$.
}

\State{Using Algorithm  12.21 (Triangular Comparison of Roots) in \cite{BPRbook2} 
order the Thom encodings in 
$\mathcal{F}$, and let the associated values be $t_1 = \ass(\tau_1),\ldots,t_N = \ass(\tau_N)$. 
\[
\mathcal{F} \gets  (\tau_1,\ldots,\tau_N). 
\]
}

\algstore{myalg}
\end{algorithmic}
\end{algorithm}
 
\begin{algorithm}[H]
\begin{algorithmic}[1]
\algrestore{myalg}

\State{Further refine $\mathcal{C}$, such that for each $\gamma \in \mathcal{C}$,
there exists $j, 1 \leq j < N$, such that 
\[
\tau_1(\gamma) = \tau_j, \tau_2(\gamma) = \tau_2.
\]
}

\State{Augment the set $\mathcal{U}$ to include the left and the right end points of each
$\gamma \in \mathcal{C}$.}

\For{ each $j, 1 \leq j \leq N-1$} 
    \State{Let $I_j$ denote a set indexing the set of curve segment representations $\gamma \in \mathcal{C}$, such that $\tau_1(\gamma) = \tau_j, \tau_2(\gamma) = \tau_{j+1}$. For $h \in I_j$, we denote the corresponding curve segment representation in $\mathcal{C}$ by $\gamma_h$.} 
    

    \State{$\mathcal{C}_j \gets (\gamma_h)_{h \in I_j}$.}

    \State{$\mathcal{U}_j \gets \{u \in \mathcal{U} \; \vert \; \pi_{i+1}(\ass(u)) = \ass(\tau_j)\}$.}
    \State{Compute the maps $L_j: I_j \rightarrow \mathcal{U}_j, R_j: I_j \rightarrow \mathcal{U}_{j+1}$, such that $\ass(L_j(h))$ is the left end-point of $\gamma_h$, and 
    $\ass(R_j(h))$ is the right endpoint of $\gamma_h$.}
\EndFor

\State{Output $\left(\mathcal{F}, (I_j,\mathcal{C}_j, \mathcal{U}_j, L_j,R_j)_{j \in [1,N]}\right)$.}

 \COMPLEXITY
The complexity of the algorithm is bounded by $ D^{O(i)}( s d)^{O((k-i)^2)}$, where $s = \card(\mathcal{P}), d = \max_{P \in \mathcal{P}} \deg(P)$, and $D = \deg(\mathcal{T})$.

The degrees of the curve segment representations in the various $\mathcal{C}_j$,
and the degrees of the real univariate representations in $\mathcal{U}_j$  are both bounded by $(D,d^{O(k-i)})$. Finally the sum of the cardinalities 
\[
\sum_{j =1}^{N} (\card(\mathcal{C}_j) + \card(\mathcal{U}_j))
\]
is bounded by  $( s d)^{O(k-i)}$.
\end{algorithmic}
\end{algorithm} 
 
\begin{proposition}
\label{prop:modified-curve-segment:correctness}
The output of Algorithm~\ref{alg:modified-curve-segment} (Curve segments) satisfies the following:
\begin{enumerate}[(a)]
    \item
    \label{itemlabel:prop:modified-curve-segment:correctness:1}
For each $j,\; 1 \leq j \leq N-1$, and all $t \in [\ass(\tau_j), \ass(\tau_{j+1}))$ the inclusion maps 
\[
S_{\{\ass(\mathcal{T})\} \times (-\infty, \ass(\tau_j)]}\hookrightarrow S_{\{\ass(\mathcal{T})\} \times (-\infty,t]},
\]
\[
S_{\{\ass(\mathcal{T})\} \times \{t\}}\hookrightarrow S_{\{\ass(\mathcal{T})\} \times [\ass(\tau_j), \ass(\tau_{j+1})},
\]
are homological equivalences;
    \item 
    \label{itemlabel:prop:modified-curve-segment:correctness:2}
    for each $h \in I_j$, $\ass(\gamma_h) \subset S = \RR(\Phi,\R^k)$;
    \item 
    \label{itemlabel:prop:modified-curve-segment:correctness:3}
    for each $t  \in (\ass(\tau_j),\ass(\tau_{j+1}))$ and each semi-algebraically connected component
    $C$ of $S_{\y}$,
    where $\y = (\ass(\mathcal{T}),t) \in \R^{i+1}$,
    there exists $h \in I_j$ such that $\ass(\gamma_h)_{\y} \in C$;
    \item 
    \label{itemlabel:prop:modified-curve-segment:correctness:4}
    if $h_1,h_2 \in I_j$ with $h_1 \neq h_2$, then 
    $\ass(\gamma_{h_1}) \cap \ass(\gamma_{h_2}) = \emptyset$.
\end{enumerate}
\end{proposition}

\begin{proof}
Part~\eqref{itemlabel:prop:modified-curve-segment:correctness:1} follows from 
the property of the output of Algorithm~\ref{alg:filtration} (Morse partition) called in Line~\ref{line:alg:surjection:filtration} (cf. Proposition~\ref{prop:alg:filtration}).

Part~\eqref{itemlabel:prop:modified-curve-segment:correctness:2} and
follows from 
\eqref{itemlabel:prop:modified-curve-segment:correctness:3}
fact 
that the output of Algorithm 16.26 (General Roadmap) in \cite{BPRbook2}
which is called in Line \ref{line:alg:modified-curve-segment:roadmap}
describes a road map of the semi-algebraic set $\RR(\Phi,\R^k)$.

Finally, 
Part~\eqref{itemlabel:prop:modified-curve-segment:correctness:4} is ensured
in Line~\ref{line:alg:surjection:intersection}.
\end{proof}

\begin{proof}[Complexity analysis of Algorithm~\ref{alg:modified-curve-segment}]
The stated complexity follows from the complexity bounds of the various algorithms used
in the algorithm, keeping in mind that each arithmetic operation in the ring
$\D[\ass(\mathcal{T})]$ costs $D^{O(i)}$ arithmetic operations in the ring $\D$.
\end{proof}

\begin{remark}
\label{rem:complexity:alg:modified-curve-segment}
Note that in the step described in Line~\ref{line:alg:modified-curve-segment:roadmap} 
it is sufficient to compute curve segment representations $\gamma$ over $\mathcal{T}$,  such that for each
$\y = (\ass(\mathcal{T},t) \in \R^{i+1}$, and each semi-algebraically connected
component $C$ of $S_\y$, there is a curve segment representation $\gamma$, such that
$\ass(\gamma)_\y \cap C \neq \emptyset$. Calling Algorithm 16.12 (Bounded Roadmap) in \cite{BPRbook2} for this purpose 
is convenient but an overkill. 
The recursive calls in constructing a full roadmap leads to the quadratic dependence 
in the exponent - however, we do not need the parts of the roadmap in the various
slices, but retain only the curve segments parametrized by $X_{i+1}$.

One could alternatively achieve the same goal by other algorithms having smaller
complexity (namely, $D^{O(i)} (sd)^{O(k-i)}$ instead of $D^{O(i)} (sd)^{O((k-i)^2)}$). 
For example, it is possible
to modify Algorithm 14.1 (Block Elimination) in \cite{BPRbook2} using as parameter
the coordinate $X_{i+1}$ for this purpose. Doing so
would reduce the total complexity of the algorithm  to $D^{O(i)} (sd)^{O(k-i)}$
instead of $D^{O(i)} (sd)^{O((k-i)^2)}$ as stated in the complexity bound. However, 
describing the modifications needed to Algorithm 14.1 (Block Elimination) in \cite{BPRbook2} would complicate the exposition and moreover will not change
the asymptotic complexity of Algorithm~\ref{alg:surjection} (Computing one-dimensional subset), which is the
only place where Algorithm~\ref{alg:modified-curve-segment} is used. So we chose 
not to make this modification.
\end{remark}

We are now in a position to describe our algorithm which will complete 
Step~\ref{itemlabel:prob-H1:1} of our proof of Theorem~\ref{thm:main}.

\begin{algorithm}[H]
\caption{(Computing one-dimensional subset)}
\label{alg:surjection}
\begin{algorithmic}[1]
\INPUT
\Statex{
\begin{enumerate}[1.]
\item
A triangular Thom encoding $\mathcal{T} = (\mathbf{F},\boldsymbol{\sigma})$ of size $i$;
\item
a finite set $\mathcal{P} \subset \D[X_1,\ldots,X_k]$;
\item
a $\mathcal{P}$-closed formula $\Phi$ such that $\RR(\Phi,\R^k)$ is bounded;
\item
a finite set $\mathcal{M}$ of real univariate representations over $\mathcal{T}$, whose set of 
associated points, $M$,  is contained in $\RR(\Phi,\R^k)$. 
\end{enumerate}
}

\OUTPUT
 \Statex{
 \begin{enumerate}[1.]
    \item a finite set $\mathcal{U}$ of real univariate representations over $\mathcal{T}$;
    \item a finite indexing set $I$ and
     a finite tuple $(\gamma_j)_{j \in I}$ where each $\gamma_j$  a curve segment representation over $\mathcal{T}$;
    \item mappings $L,R:I \rightarrow \mathcal{U}$, defined by $\ass(L(j))$ is the left endpoint of $\gamma_j$, and $\ass(R(j))$ is the right endpoint of $\gamma_j$.
 \end{enumerate}
}

\PROCEDURE
\If {$k-i=1$}
    \For {each $u = \{(f, g_0,g_k), \sigma\}  \in \mathcal{M}$} 
     \label{itemlabel:alg:surjection:line:includepoints1}
        \State{Let $R_u \in \R[X_k]$ be the Sylvester resultant (see for example \cite[pp. 106]{BPRbook2}) with respect to the variable $T$ of the polynomials
        $f, X_k g_0 - g_k$.}
        \State{Use Algorithm 10.11 (Sign Determination) \cite[pp. 390]{BPRbook2} 
        to compute a Thom encoding $\tau_u = (R_u,\sigma_u)$ over $\mathcal{T}$, such that $\ass(\tau_u) = \pi_k(\ass(u))$.}
        \label{line:alg:surjection:1}
        \State{$\mathcal{P} \gets \mathcal{P} \cup \{R_u\}$.}
    \EndFor
    
    \State{
    Use Algorithm 12.21 (Triangular Comparison of Roots) in \cite[pp. 496]{BPRbook2}
    repeatedly with inputs $\mathcal{T}$ and pairs of polynomials in $\mathcal{P}$,
    and order the real roots of the polynomials $P(\ass(\mathcal{T}), X_k), P \in \mathcal{P}$, and hence  obtain a partition of 
    $\R$ into points and open intervals, and identify those points and open intervals which are contained in $\RR(\Phi,\R^k)_{\ass(\mathcal{T})}$.}
    
    \State{$\mathcal{U} \gets \emptyset$.}
    \State{$I \gets \emptyset$.}
    \State{$j \gets 0$.}

    \For {each 
    Thom encoding $(P,\sigma)$ over $\mathcal{T}$ obtained in Line~\ref{line:alg:surjection:1}  whose associated point is in
    $S = \RR(\Phi,\R^k)$} 
    \State{$\mathcal{U} \gets \mathcal{U} \cup \{ ((P, X_k, 1), \sigma)\}$.}
    \EndFor

\algstore{myalg}
\end{algorithmic}
\end{algorithm}
 
\begin{algorithm}[H]
\begin{algorithmic}[1]
\algrestore{myalg}

    \For{each open interval with end-points described by the Thom encodings 
    $\tau_1 = ((P_1, X_k, 1), \sigma_1), \tau_2 = ((P_2, X_k, 1), \sigma_2) \in \mathcal{U}$
    with $\ass(\tau_1) < \ass(\tau_2)$ such that 
    $(\ass(\tau_1),\ass(\tau_2)) \subset \pi_k(\RR(\Phi,\R^k))$
    }

    \State{$I \gets I \cup \{j\}$.}
    \State{$j \gets j+1$.}

    \State{$\gamma_j \gets \gamma$, where $\gamma$ is the curve segment representation over $\mathcal{T}$ defined by:
    \begin{eqnarray*}
    \tau_1(\gamma) &=& \tau_1, \\ 
    \tau_2(\gamma) &=& \tau_2, \\ 
    u(\gamma) &=& ((T,1),(0,1)).
    \end{eqnarray*}
    }

    \State{$L(j) \gets  \tau_1(\gamma_j)$.} 
    \State{$R(j) \gets  \tau_2(\gamma_j)$.} 
    \EndFor

    \State{Output $\mathcal{U}, (\gamma_j)_{j \in I}$,  and the mappings $L,R:I\rightarrow \mathcal{U}$.}
\EndIf

\State{
\label{itemlabel:alg:surjection:line:curve-segment}
Use Algorithm~\ref{alg:modified-curve-segment} (Curve segments)
with input $(\mathcal{T}, \mathcal{P},\Phi)$ to compute: 

\begin{enumerate}[(a)]
    \item A finite tuple  $\mathcal{F} =(\tau_1,\ldots,\tau_N)$ of Thom encodings over $\mathcal{T}$, with
    $\ass(\tau_1) < \cdots < \ass(\tau_N)$;
    \item for each $j, 1 \leq j \leq N-1$,
    a finite tuple $\mathcal{C}_j$ of curve segment representations over $\mathcal{T}$ 
    such that for each $\gamma \in \mathcal{C}_j$,  $\tau_1(\gamma) = \tau_j, \tau_2(\gamma) = \tau_{j+1})$;
    \item for each $j, 1 \leq j \leq N$
a finite set $\mathcal{U}_j$ of real univariate representations over $\mathcal{T}$, such that for each
$u \in \mathcal{U}_j$, 
 the set of points $\{\ass(u) \;\vert\; u \in \mathcal{U}_j \}$ is precisely the set of end-points of the curve segments in $\mathcal{C}_{j-1} \cup \mathcal{C}_{j}$ (with the convention that $\mathcal{C}_{0} = \emptyset$);
    \item mappings $L_j,R_{j-1}: \mathcal{C}_j \rightarrow \mathcal{U}_j$, such that
    $\ass(L_j(\gamma))$ is the left end-point of $\gamma$, and $\ass(R_j(\gamma))$ is the right end-point of $\gamma$.
\end{enumerate}
}

\For {$\tau = (f,\sigma)  \in \mathcal{F}$}
         \State{$\mathcal{M}_\tau \gets \{u \in \mathcal{M} \cup \mathcal{U} \;\vert\; 
         \pi_{i+1}(\ass(u)) = \ass(\tau) \}$.} 
          \label{itemlabel:alg:surjection:line:includepoints2}
         \State{$\mathcal{T}_\tau \gets ((F,f), (\boldsymbol{\sigma},\sigma)))$.}
         \State{Call Algorithm~\ref{alg:surjection} (Computing one-dimensional subset) recursively with input  ($ \mathcal{T}_\tau,\mathcal{P}, \Phi, \mathcal{M}_\tau$) and obtain 
         a set of $\mathcal{U}_\tau$ of real univariate representations over $\mathcal{T}_\tau$, an indexing set $I_\tau$, a tuple $(\gamma_i)_{i \in I_\tau}$ of curve segment representations, and mappings
         $L_\tau,R_\tau: I_\tau \rightarrow \mathcal{U}_\tau$.
         (Note that for each $i \in I_\tau, \ass(\gamma_i) \subset S_{\ass(\tau)}$.)}
         \label{itemlabel:alg:surjection:line:recurse}
\EndFor
\State{$I \gets \bigcup_{\tau \in \mathcal{F}} I_\tau$.}
\State{$\mathcal{U}  \gets \bigcup_{\tau \in \mathcal{F}} \mathcal{U}_\tau$.}
\State{$L \gets \bigcup_{1 \leq j \leq N } L_j \cup \bigcup_{\tau \in \mathcal{F}} L_\tau$.
(Union of disjoint mappings means the disjoint union of their graphs.)}
\State{$R \gets \bigcup_{1 \leq j \leq N } R_j \cup \bigcup_{\tau \in \mathcal{F}} R_\tau$.}
\State{Output $\mathcal{U},I, (\gamma_j)_{j \in I}, L,R$.}
 \COMPLEXITY
Suppose that $\deg(\mathcal{T}) = d^{O(k)}$.
The complexity of the algorithm is bounded by 
$(s d)^{O(k^2)}$, where $s = \card(\mathcal{P}), d = \max_{P \in \mathcal{P}}$.
Moreover,
$\card(I) = ( s d)^{O((k-i)^2)}$, and 
the degrees of the elements of $\mathcal{U}$ and $\gamma_j, j \in I$
are bounded by $(d^{O(k)}, d^{O(k)})$.
\end{algorithmic}
\end{algorithm}
 
 \begin{proposition}
 \label{prop:alg:surjection:correctness}
 The output of Algorithm~\ref{alg:surjection} (Computing one-dimensional subset) has the following properties.
 Let $\Gamma = \bigcup_{j \in I} \clos(\ass(\gamma_j))$, and
 $ S = \RR(\Phi,\R^k)_{\ass(\mathcal{T})}$.
\begin{enumerate}[(a)]
    \item 
    \label{itemlabel:prop:alg:surjection:correctness:1}
        $M \subset \Gamma$;
    \item 
    \label{itemlabel:prop:alg:surjection:correctness:2}
        $\Gamma \subset S$;
    \item 
    \label{itemlabel:prop:alg:surjection:correctness:3}
        $\dim(\Gamma) \leq 1$;
    \item 
    \label{itemlabel:prop:alg:surjection:correctness:4}
        the homomorphism  $i_{*,0}:\HH_0(\Gamma) \rightarrow \HH_0(S)$ induced by the inclusion map $i:\Gamma \hookrightarrow S$ is an isomorphism;
    \item 
    \label{itemlabel:prop:alg:surjection:correctness:5}
        the homomorphism $i_{*,1}:\HH_1(\Gamma) \rightarrow \HH_1(S)$ 
        induced by the inclusion map $i:\Gamma \hookrightarrow S$ is an epimorphism.
\end{enumerate}
\end{proposition}
 
 \begin{proof}
 The property in Part \eqref{itemlabel:prop:alg:surjection:correctness:1} is ensured in
 Lines~\ref{itemlabel:alg:surjection:line:includepoints1}
 and \ref{itemlabel:alg:surjection:line:includepoints2}.
 Part
 \eqref{itemlabel:prop:alg:surjection:correctness:2}
 follows from 
 the property of the output of Algorithm~\ref{alg:modified-curve-segment} (Curve segments) (called in 
Line~\ref{itemlabel:alg:surjection:line:curve-segment})
given in 
Part \eqref{itemlabel:prop:modified-curve-segment:correctness:2} of
Proposition~\ref{prop:modified-curve-segment:correctness}.

Part \eqref{itemlabel:prop:alg:surjection:correctness:3} 
is clear since $\Gamma$ is by definition the finite union
\[
\bigcup_{j \in I} \clos(\ass(\gamma_j))
\]
and $\dim \ass(\gamma_j) = 1$ for each $j \in I$, and
taking the closure does not increase the dimension of a semi-algebraic set.
 
 It is a standard exercise (see for example proof of Proposition 15.7 in \cite{BPRbook2}) to prove  that the semi-algebraic set $\Gamma$  
 satisfies the properties
 of being roadmap for $S = \RR(\Phi,\R^k)$ (see \cite[Chapter 15]{BPRbook2}), which
 implies Part~\eqref{itemlabel:prop:alg:surjection:correctness:4}.

 We now prove Part~\eqref{itemlabel:prop:alg:surjection:correctness:5}.
 
 The proof is by induction on $k - i$.
 
 Base case: $k-i = 1$. In this case the claim is clear since $\HH_1(\Gamma) = \HH_1(S) = 0$.
 
 Suppose the claim is true for all smaller values of $k-i$. Notice that Algorithm~\ref{alg:surjection} (Computing one-dimensional subset) is called recursively in Line~\ref{itemlabel:alg:surjection:line:recurse}. In these calls
the triangular Thom encoding $\mathcal{T}_\tau, \tau \in \mathcal{F}$ in the input
 is of size $i+1$, while the number of variables is still $k$.
 We have also have,
 \[
 \Gamma_\tau := \clos\left(\bigcup_{j \in I_\tau} \ass(\gamma_j)\right) =  \Gamma_{\ass(\tau)}.
 \]
 Thus, using the induction hypothesis for this recursive call (since $k- (i+1) < k- i$),
 we obtain that the restriction of the inclusion $\Gamma \rightarrow S$ to
 $\Gamma_\tau$ induces a surjection
 \begin{equation}
 \label{eqn:prop:alg:surjection:correctness:proof:0}
 \HH_1(\Gamma_\tau) \rightarrow \HH_1(S_{\ass(\tau)}).
 \end{equation}
 
 Denote for $1 \leq j \leq N$, 
 $t_j = \pi_{i+1}(\ass(\tau_j))$.
 
 \begin{eqnarray*}
 a_0 &=& t_1, \\
 a_j &=& \frac{t_j + t_{j+1}}{2}, 1 \leq j < N,\\
 a_N &=& t_N.
 \end{eqnarray*}
We prove the following claims.

\begin{claim}
For each $j, 0 \leq j \leq N$, the inclusion maps 
induce the following isomorphisms.
\begin{align}
\label{eqn:prop:alg:surjection:correctness:proof:1}
 \HH_*(S_{(-\infty,t_j]}) \rightarrow \HH_*(S_{(-\infty,a_j]}), \\
 \label{eqn:prop:alg:surjection:correctness:proof:2}
 \HH_*(S_{t_j}) \rightarrow \HH_*(S_{[a_{j-1},a_j]}), \\
 \label{eqn:prop:alg:surjection:correctness:proof:3}
 \HH_*(\Gamma_{(-\infty,t_j]}) \rightarrow \HH_*(\Gamma_{(-\infty,a_j]}), \\
 \label{eqn:prop:alg:surjection:correctness:proof:4}
 \HH_*(\Gamma_{t_j}) \rightarrow \HH_*(\Gamma_{[a_{j-1},a_j]}).
\end{align}
\end{claim}

\begin{proof}
Parts \eqref{eqn:prop:alg:surjection:correctness:proof:1} and
\eqref{eqn:prop:alg:surjection:correctness:proof:2} 
are consequences of the 
property of the output of Algorithm~\ref{alg:modified-curve-segment} (Curve segments)
(which is called in 
Line~\ref{itemlabel:alg:surjection:line:curve-segment})
given in 
Part \eqref{itemlabel:prop:modified-curve-segment:correctness:1} of
Proposition~\ref{prop:modified-curve-segment:correctness}.

Parts \eqref{eqn:prop:alg:surjection:correctness:proof:3} and
\eqref{eqn:prop:alg:surjection:correctness:proof:4} follow from the fact that
there is an easy to define retraction (along the $X_{i+1}$-coordinate) of 
$\Gamma_{(-\infty,a_j]}$ to $\Gamma_{(-\infty,t_j]}$ (resp. 
$\Gamma_{[a_{j-1},a_j]}$ to $\Gamma_{t_j}$) making use of the fact that
distinct curve segments over the open intervals $(t_{j-1},t_j)$ do not intersect
which is ensured by
Part \eqref{itemlabel:prop:modified-curve-segment:correctness:4} of
Proposition~\ref{prop:modified-curve-segment:correctness}.
\end{proof}

\begin{claim}
\label{prop:alg:surjection:correctness:proof:claim:2}
Let $a,b \in \{a_0,\ldots, a_N\}$, with $a \leq b$.
The inclusion map $\Gamma \hookrightarrow S$ induces 
isomorphisms 
\[
\HH_0(\Gamma_{(-\infty,a]}) \rightarrow \HH_0(S_{(\infty,a]}),
\]
\[
\HH_0(\Gamma_{[a,b]}) \rightarrow \HH_0(S_{[a,b]}),
\]
\end{claim}
\begin{proof}
This follows from the fact that $\Gamma_{(-\infty,a]}$ (resp. $\Gamma_{[a,b]}$)
satisfy the roadmap property with respect to the set $S_{(\infty,a]}$ 
(resp. $S_{[a,b]}$). The proof of this fact is standard and omitted.
\end{proof}

Using the claims proved above, we are now going to prove using induction on $j$, 
that 
the inclusion map $\Gamma \rightarrow S$ induces an isomorphism,
\[
\HH_1(\Gamma_{(\infty,a_j]}) \rightarrow \HH_1(S_{(\infty,a_j]}).
\]

The claim is true for $j=0$ by the global induction hypothesis on $i$, and 
hence is also true for $j=1$ using using \eqref{eqn:prop:alg:surjection:correctness:proof:1} and \eqref{eqn:prop:alg:surjection:correctness:proof:3}.

We prove it for $j>1$ by induction. Suppose the claim holds until $j-1$. 

Hence we have isomorphism
\[
\HH_1(\Gamma_{(\infty,a_{j-1}]}) \rightarrow \HH_1(S_{(\infty,a_{j-1}]})
\]
induced by inclusion.

Observe that for any set $X \subset \R^k$,
\begin{eqnarray*}
X_{(-\infty,a_j]} &=& X_{(-\infty,a_{j-1}]} \cup X_{[a_{j-1},a_j]}, \\
X_{a_{j-1}}&=& X_{(-\infty,a_{j-1}]} \cap X_{[a_{j-1},a_j]}  X_{a_{j-1}}.
\end{eqnarray*}

Let
\begin{eqnarray*}
A_1 &=& \Gamma_{(-\infty,a_{j-1}]},\\
A_2 &=& \Gamma_{[a_{j-1},a_j]}, \\
B_1 &=& S_{(-\infty,a_{j-1}]},\\
B_2 &=& S_{[a_{j-1},a_j]}.\\
\end{eqnarray*}
Also, 
let $A_{12}$ (resp. $B_{12}$) denote $A_1\cap A_2 $ (resp. $B_1 \cap B_2$), and
$A^{12}$ (resp. $B^{12}$) denote $A_1\cup A_2 $ (resp. $B_1 \cup B_2$).

The Mayer-Vietoris exact sequence (see for example \cite[Theorem 6.35]{BPRbook2}) yields the following commutative diagrams with exact rows and vertical arrows induced by various restrictions of the  inclusion $\Gamma \hookrightarrow S$.

\[
\xymatrix{
\HH_1(A_{12}) \ar[d]^a\ar[r]& \HH_1(A_1)\oplus\HH_1(A_2)\ar[d]^{b_1\oplus b_2}\ar[r]& \HH_1(A^{12})\ar[d]^c\ar[r]
& \HH_0(A_{12})\ar[d]^d\ar[r] &\HH_0(A_1) \oplus \HH_0(A_2)\ar[d]^{e_1\oplus e_2} \\
\HH_1(B_{12})\ar[r]& \HH_1(B_1)\oplus\HH_1(B_2)\ar[r]& \HH_1(B^{12})\ar[r]
& \HH_0(B_{12})\ar[r] &\HH_0(B_1) \oplus \HH_0(B_2) 
}
\]

By induction hypothesis on $j$, 
the map $b_1:\HH_1(A_1) \rightarrow \HH_1(B_1)$ is surjective.
Using 
\eqref{eqn:prop:alg:surjection:correctness:proof:0},
\eqref{eqn:prop:alg:surjection:correctness:proof:2},
 and \eqref{eqn:prop:alg:surjection:correctness:proof:4} 
we have that the map $b_2:\HH_1(A_2) \rightarrow \HH_1(B_2)$ is surjective.
Hence the map $b = b_1\oplus b_2$ is surjective.

Using 
Proposition~\ref{prop:modified-curve-segment:correctness} ~\eqref{itemlabel:prop:modified-curve-segment:correctness:3}
we have that map $d$ is surjective.

Finally, using Claim~\ref{prop:alg:surjection:correctness:proof:claim:2}, we have that
the maps $e_1$ and $e_2$ are both isomorphisms, and hence so is $e$. In particular,
$e$ is injective.

It follows from the above and 
(one-half of)  the ``Five-lemma" (see for example \cite{Bourbaki}),
that $c$ is a surjection.
 \end{proof}

\begin{proof}[Complexity analysis of Algorithm~\ref{alg:surjection}]
It follows from the complexity bounds on the algorithms used that all the steps
before the recursive call in Line~\ref{itemlabel:alg:surjection:line:recurse} has complexity bounded by
$(s d)^{O((k-i)^2)}$ in terms of the number of arithmetic operations in $\D[\ass(\mathcal{T}]$. Hence,the number of arithmetic operations in $\D$ for these
steps is bounded by $ d^{O(k i)} (sd)^{O((k-i)^2)}$.

There are $(sd)^{O(k-i)}$ recursive calls. For each of the recursive calls,
the new triangular Thom encoding is of size $i+1$, and its degree is bounded again
by $d^{O(k)}$. An easy inductive argument now implies the stated complexity bounds.
\end{proof} 

\subsection{Implementing Step~\ref{itemlabel:prob-H1:2}: proofs of Theorems~\ref{thm:main} and \ref{thm:basic}}
\label{subsec:basis}

We will now 
prove Theorem~\ref{thm:main} by describing 
an algorithm
(cf. Algorithm~\ref{alg:basis} below) for computing 
a semi-algebraic basis of $\HH_1(\RR(\Phi,\R^k))$, for any given closed
formula $\Phi$. Theorem~\ref{thm:main} will then follow 
from the proof of  correctness this algorithm  and the analysis of the complexity
of the algorithm.

\subsubsection{Outline of Algorithm~\ref{alg:basis}}
Our main tool will be Algorithm~\ref{alg:surjection} (Computing one-dimensional subset) that produces 
a one-dimensional subset $\Gamma$ of $S = \RR(\Phi,\R^k)$, such that the image 
of $\HH_1(\Gamma)$ in $\HH_1(S)$ under the linear map induced by inclusion is 
surjective. The semi-algebraic set $\Gamma$ has an underlying structure of 
a finite graph $G$, and $\Gamma$ is semi-algebraically  homeomorphic to the geometric realization $|G|$
of $G$. Using a combinatorial graph-theoretic algorithm it is easy to 
compute a basis of the cycle space, $Z(G) \cong \HH_1(\Gamma)$, consisting of 
simple cycles (say) $C_1,\ldots,C_N$. Each such simple cycle $C_i$ is a subgraph of $G$,
and its geometric realization $|C_i| \subset |G| \cong \Gamma$ is semi-algebraically homeomorphic to $\Sphere^1$. 

We will denote by $[C_i]$ a non-zero element of the image of $\HH_1(|C_i|)$ in $\HH_1(S)$
if this image is non-zero (and let $[C_i] = 0$ otherwise).
However, the set $\{[C_1],\ldots,[C_N]\}$ need not be linearly independent but thanks to 
the surjectivity of the map $\HH_1(\Gamma) \rightarrow \HH_1(S)$, they span $\HH_1(S)$.
We identify a minimal subset of $\{[C_1],\ldots,[C_N]\}$ which spans $\HH_1(S)$ (i.e. form a basis of $\HH_1(S)$). The corresponding subsets of 
$\Gamma$ then constitutes a semi-algebraic basis of $\HH_1(S)$.
In order to achieve this last step effectively and with singly exponential complexity, we
use a recent result proved in \cite{Basu-Karisani}, giving an algorithm
(see Algorithm~\ref{alg:simplicial-replacement} below for input, output and complexity)
with singly exponential complexity for replacing a tuple of given closed and bounded semi-algebraic subsets of $\R^k$, by a simplicial complex which is homologically $1$-equivalent to the
given tuple of sets (cf. Definition~\ref{def:equivalence-diagrams}). This simplicial complex is of singly exponential bounded size, and the various $|C_i|$'s can be identified with subcomplexes of this complex. So to find a basis from amongst the  elements
$[C_1],\ldots,[C_N]$ becomes a problem of ordinary linear algebra which can be solved 
with polynomial complexity.

\subsubsection{Conversion of curve segment representations to closed formulas}
We will need the following algorithm.
It is needed in order to address the following technical issue.

The output of Algorithm~\ref{alg:surjection} (Computing one-dimensional subset)
contains amongst other objects, a set of curve segment representations. 
We will need to convert these descriptions into \emph{closed formulas} describing the 
closure of the associated curves in order to use 
Algorithm~\ref{alg:simplicial-replacement}
which only accepts such descriptions in the input.

Note that the algorithmic problem of computing a \emph{closed}
formula describing the closure of a given  
semi-algebraic set described by a quantifier-free (but not necessarily closed-) formula
is far from being easy, and no algorithm with singly exponential complexity is known for solving this problem in general.
\footnote{Note that without the requirement that the output formula be closed it is straight-forward to obtain a 
singly exponential complexity algorithm via quantifier elimination.}
(A doubly exponential algorithm is known, using the notion
of a stratifying family \cite[Chapter 5]{BPRbook2}).
However, fortunately for us the curve segment representations describing the associated curve have a special structure. In particular, it is clear that given a curve segment representation $\gamma$,  it is algorithmically quite simple to obtain a description of the the image,
$\pi_{\{1,j\}}(\ass(\gamma))$,
of the projection of $\ass(\gamma)$, to each of the coordinate subspaces spanned by $(X_1,X_j), 2 \leq j \leq k$. 
More precisely,
suppose that $\gamma$ is the curve segment representation with
\[
u(\gamma) = ((f, g_0, \ldots, g_k), \sigma),
\]
where 
$f,g_i \in \D[X_1,T]$, and $\sigma \in \{0,1,-1\}^{\Der_T(f)}$. Then, $\ass(\gamma)$
is defined by
\[
\ass(\gamma) = \left\{ \left(x_1, \frac{g_2(x_1,t(x_1))}{g_0(x_1,t(x_1))},\ldots, \frac{g_k(x_1,t(x_1))}{g_0(x_1,t(x_1))}\right) \;\middle\vert\;
\ass(\tau_1(\gamma)) < x_1 < \ass(\tau_2(\gamma)) \right\},
\]
where for each $x_1, \ass(\tau_1(\gamma)) < x_1 < \ass(\tau_2(\gamma))$, and $t(x_1)$ is
a root of $f(x_1,T)$ with Thom encoding $\sigma$.

Now for each $j, 2 \leq j \leq k$, the projection of $\ass(\gamma)$ to the $(X_1,X_j)$-plane
is  described by
\[
\pi_{\{1,j\}}(\ass(\gamma)) = \left\{ \left(x_1, \frac{g_j(x_1,t(x_1))}{g_0(x_1,t(x_1))}\right)  \;\middle\vert\;
\ass(\tau_1(\gamma)) < x_1 < \ass(\tau_2(\gamma))\right\}.
\]

Using an effective quantifier elimination (eliminating $T$), one can obtain from the above description a quantifier-free formula with free variables $X_1,X_j$ whose realization is
equal to $\pi_{\{1,j\}}(\ass(\gamma))$.

We will use the following claim.
\begin{claim}
\label{claim:alg:convert-to-closed}
Suppose that $\gamma$ is a curve segment representation and $\ass(\gamma)$ is bounded, then
\[
\clos(\ass(\gamma)) = \bigcap_{2 \leq j \leq k} \pi^{-1}_{\{1,j\}}(\clos(\pi_{\{1,j\}}(\ass(\gamma)))).
\]
\end{claim}
\begin{proof}[Proof of Claim~\ref{claim:alg:convert-to-closed}]
First, suppose that $\x \in \clos(\ass(\gamma))$. Then we have that 
\[
\x \in \pi^{-1}_{\{1,j\}}(\pi_{\{1,j\}}(\ass(\gamma)))
\]
for all $j$. 
It is a general fact from topology that for any continuous function $f$ and set $A$, $f(\clos(A)) \subset \clos(f(A))$. Hence, \[
\pi^{-1}_{\{1,j\}}(\pi_{\{1,j\}}(\ass(\gamma))) \subset \pi^{-1}_{\{1,j\}}(\clos(\pi_{\{1,j\}}(\ass(\gamma))))
\]
for all $j$. Therefore $\x \in \bigcap_{2 \leq j \leq k} \pi^{-1}_{\{1,j\}}(\clos(\pi_{\{1,j\}}(\ass(\gamma))))$.

Now, suppose that 
\[
\x  = (x_1,\ldots,x_k) \in \bigcap_{2 \leq j \leq k} \pi^{-1}_{\{1,j\}}(\clos(\pi_{\{1,j\}}(\ass(\gamma))))
\]
and let 
$\x_{1,j} = \pi_{\{1,j\}}(\x)$
for $j = 2,\ldots, k$.

Since, $\x_{1,j} \in \clos(\pi_{\{1,j\}}(\ass(\gamma)))$, 
using the semi-algebraic curve selection lemma (see for example \cite[Theorem 3.19]{BPRbook2}) there exists $t_{j,0} > 0$, 
such that there exists a semi-algebraic curve, 
$\gamma_j = (\gamma_{1,j},\gamma_{2,j}) :[0, t_{j,0}]: \R^2$, such that
$\gamma_{1,j},\gamma_{2,s}$ are continuous semi-algebraic functions,
$\gamma_j(0) = \x_{1,j}$,  and $\gamma((0,t_{j,0}]) \subset \pi_{\{1,j\}}(\ass(\gamma))$.
Moreover, clearly since $\pi_{\{1,j\}}(\ass(\gamma))$ is a curve parametrized by the 
$X_1$ coordinate, $\gamma_{1,j}$ is not a constant function, and without loss of
generality (choosing $t_{j,0}$ smaller if necessary) we can assume that
$\gamma_{1,j}$ is a strictly increasing function.
For $j=2,\ldots,k$,
let $f_j:[x_1, \gamma_{1,j}(t_{j,0})] \rightarrow \R$ be defined by
\[
f_j(X_1) =\gamma_{2,j}(\gamma_{1,j}^{-1}(X_1)). 
\]
Taking $x_1' = \min_{2 \leq j \leq k} \gamma_{1,j}(t_{j,0})$, we obtain a semi-algebraic curve
$\widetilde{\gamma}: [x_1,x_1'] \rightarrow \ass(\gamma)$, defined by
\[
\widetilde{\gamma}(X_1) = (X_1, f_2(X_1),\ldots, f_k(X_1)).
\]
It is easy to check that
\[
\widetilde{\gamma}(x_1) = \x,
\]
and 
\[
\widetilde{\gamma}: (x_1,x_1'] \subset \ass(\gamma),
\]
which proves that $\x \in \clos(\ass(\gamma))$.
\end{proof}

Using Claim~\ref{claim:alg:convert-to-closed}  we reduce the problem of computing a closed description of $\clos(\ass(\gamma))$ to the problem of computing the closures 
of $\pi_{\{1,j\}}(\ass(\gamma))$, $2 \leq j \leq k$, and each of the latter is a
$2$-dimensional problems which can be solved within our allowed complexity bound using the doubly exponential algorithm referred to previously.

Finally, as in the case of the other algorithms in this paper we include in
the input a triangular Thom encoding $\mathcal{T}$ 
that fixes the first $i$-coordinates, and 
the curve segment representations in the input is over 
$\mathcal{T}$. The computations in the algorithm takes place in the ring
$\D[\ass(\mathcal{T})]$, and in the description given above, the first coordinate
is replaced by the $(i+1)$-st coordinate.
 
\begin{algorithm}[H]
\caption{(Conversion of curve segment representations to closed formulas)}
\label{alg:convert-to-closed}
\begin{algorithmic}[1]
\INPUT
\Statex{
\begin{enumerate}[1.]
\item
A triangular Thom encoding 
$\mathcal{T} = (\mathbf{F},\boldsymbol{\sigma})$ of size $i, 0 \leq i \leq k$;
\item
a curve segment representation $\gamma$ over $\mathcal{T}$.
\end{enumerate}
}
\OUTPUT
\Statex{
\begin{enumerate}[1.]
    \item A finite set of polynomials $\mathcal{Q} \subset \D[\ass(\mathcal{T}][X_{i+1},\ldots,X_k]$;
    \item A $\mathcal{Q}$-closed formula $\Psi$ such that 
    \[
    \RR(\Psi,\R^k)_{\ass(\mathcal{T})}  = \clos(\ass(\gamma)).
    \]
\end{enumerate}
}

\PROCEDURE
\State{$u \gets u(\gamma) = ((f,g_0,g_{i+2},\ldots,g_k), \sigma)$.}
\For{$j=i+2,\ldots, k$}
\label{alg:convert-to-closed:line:1}
\State{Using Algorithm 14.5 (Quantifier Elimination) in \cite[pp. 549]{BPRbook2} with the formula
\[
(\exists T) (f = 0) \wedge \bigwedge_{1 \leq h \leq \deg_T(f)} 
(\sign(f^{(h)}) = \sigma(f^{(h)})) \wedge (X_{j} g_0  - g_j = 0),
\]
as input,
and obtain a 
$\widetilde{\mathcal{Q}}_j$-formula 
quantifier-free formula $\widetilde{\phi}_j$, 
for some $\widetilde{\mathcal{Q}}_j \subset \D[\ass(\mathcal{T})][X_{i+1}, X_j]$ 
describing
such that
$ \RR(\widetilde{\phi}_j) = \pi_{[1,i+1] \cup \{j\}}(\ass(\gamma))$.
\label{alg:convert-to-closed:line:2}
}

\State{Compute a stratifying family of polynomials (see \cite[Proposition 5.40]{BPRbook2} for definition),
$\mathcal{Q}_j \subset \D[\ass(\mathcal{T})][X_{i+1}, X_j]$
containing $\widetilde{\mathcal{Q}}_j$.
}
\label{alg:convert-to-closed:line:3}

\State{Using Algorithm 13.1 (Computing realizable sign conditions) in \cite[pp. 511]{BPRbook2} 
determine the set $\Sigma_j \subset \{0,1,-1\}^{\mathcal{Q}_j}$
of realizable sign conditions of $\mathcal{Q}_j$.}
\label{alg:convert-to-closed:line:4}
\State{Determine $\Theta_j \subset \Sigma_j$ such that
$\pi_{[1,i+1] \cup \{j\}}(\ass(\gamma)) = \bigcup_{\theta \in \Theta_j} \RR(\theta)$.}

\State{$\Psi_j \gets \bigvee_{\theta \in \Theta_j} \overline{\theta}$.}
\label{alg:convert-to-closed:line:5}
\EndFor

\State{$\mathcal{Q} \gets \mathbf{F} \cup \bigcup_{i+2 \leq j \leq k} \mathcal{Q}_j$.} 
\State{$\Psi \gets  \overline{\boldsymbol{\sigma}} \wedge \bigwedge_{i+2 \leq j \leq k} \Psi_j$ (see Notation~\ref{not:weak-Thom encoding} for definition of $\overline{\boldsymbol{\sigma}}$).}
\label{alg:convert-to-closed:line:6}
\State{Output $\mathcal{Q}, \Psi$.}

\COMPLEXITY
The complexity of the algorithm is bounded by $(k - i) D_1^{O(i)} D_2^{O(1)}$, where
$D_1 = \deg(\mathcal{T})$, and $D_2 = \deg(\gamma)$. Moreover, $\card(\mathcal{Q})$
is bounded by $i+ (k - i)D_2^{O(1)}$, and 
the degrees of the polynomials in $\mathcal{Q}$ are bounded by $\max(D_1,D_2^{O(1)})$.
\end{algorithmic}
\end{algorithm} 

\begin{proof}[Proof of Correctness of Algorithm~\ref{alg:convert-to-closed}]
The algorithm reduces the problem to obtaining closed formulas describing the closures of the various $\pi_{i+1,j}(\ass(\gamma)), i+2 \leq j \leq k$ (cf. Line~\ref{alg:convert-to-closed:line:1}). After obtaining the descriptions of
the various $\pi_{i+1,j}(\ass(\gamma))$ using effective quantifier elimination algorithm
(Algorithm 14.5 (Quantifier Elimination) in \cite[pp. 549]{BPRbook2} called in 
Line~\ref{alg:convert-to-closed:line:2}), closed formulas are obtained
describing the closure by computing a stratifying family (cf. \ref{alg:convert-to-closed:line:3}) in each case. The important property of the
stratifying families $\mathcal{Q}_j$ is that the closures, $\clos(\pi_{i+1,j}(\ass(\gamma)))$ are unions of realizations of a set of weak sign conditions on $\mathcal{Q}_j$. This is a consequence of the generalized Thom's Lemma
(see \cite[Proposition 5.39]{BPRbook2}). Finally, the set of weak sign conditions
on $\mathcal{Q}_j$ whose realizations are contained in $\clos(\pi_{i+1,j}(\ass(\gamma)))$ computed using Algorithm 13.1 (Computing realizable sign conditions) in \cite[pp. 511]{BPRbook2} in Line~\ref{alg:convert-to-closed:line:4}.
The disjunction of these weak formulas gives a closed formula, $\Psi_j$ describing
$\clos(\pi_{i+1,j}(\ass(\gamma)))$ (cf. Line~\ref{alg:convert-to-closed:line:5}).
Now Claim~\ref{claim:alg:convert-to-closed} together with Lemma~\ref{lem:Thom}
(and noting that the conjunction of a finite set of closed formulas is also closed),
imply that
the conjunction, $\Psi$, of the formulas $\Psi_j, j=i+2, \ldots,k$ along with the closed formula $\overline{\boldsymbol{\sigma}}$ (Notation~\ref{not:weak-Thom encoding}), describes $\clos(\ass(\gamma))$ (cf. Line~\ref{alg:convert-to-closed:line:6}).
This proves the correctness of the algorithm.
\end{proof}

\begin{proof}[Complexity analysis of Algorithm~\ref{alg:convert-to-closed}]
As before, each arithmetic operation in $\D[\ass(\mathcal{T})]$ costs $D_1^{O(i)}$ arithmetic operations in $\D$ (where $D_1 = \deg(\mathcal{T})$).
There are $(k-i)$ two dimensional projections (cf. Line~\ref{alg:convert-to-closed:line:1}). The complexity of each of these two-dimensional 
sub-problems (measured in terms of number of operations in $\D[\ass(\mathcal{T})]$)
is bounded by $D_2^{O(1)}$, where $D_2 = \deg(\gamma)$, and this follows from the complexity bounds on the various algorithms used in the different steps
(namely, Algorithm 14.5 (Quantifier Elimination) in \cite[pp. 549]{BPRbook2}
in Line~\ref{alg:convert-to-closed:line:2}, 
algorithm for computing stratifying families in \ref{alg:convert-to-closed:line:3}, and
Algorithm 13.1 (Computing realizable sign conditions) in \cite[pp. 511]{BPRbook2} in Line~\ref{alg:convert-to-closed:line:4}). Note that these algorithms are used 
with the number of variables equal to $2$,  and hence the complexity of each call
(measured in terms of arithmetic operations in $\D[\ass(\mathcal{T})]$)
are polynomially bounded in $D_1$. This completes the complexity analysis of 
Algorithm~\ref{claim:alg:convert-to-closed}.
\end{proof}

\subsubsection{Efficient algorithm for computing a simplicial replacement}
As explained in the outline above, in order to obtain a semi-algebraic basis 
we will use an algorithm described in \cite{Basu-Karisani}.
We reproduce below  for the reader's benefit the input, output and complexity
of this algorithm. But before stating these we need some preliminary definitions.

\begin{definition}[Homological $\ell$-equivalence]
\label{def:equivalence-spaces}
We say that a map $f:X \rightarrow Y$ between two topological spaces is a \emph{homological $\ell$-equivalence}
if the induced homomorphisms between the homology groups $f_*: \HH_i(X) \rightarrow \HH_i(Y)$ are isomorphisms for $0 \leq i \leq \ell$.
\end{definition}

The relation of  homological $\ell$-equivalence as defined above is not an equivalence relation since it is not 
symmetric. In order to make it symmetric one needs to ``formally invert'' $\ell$-equivalences.

\begin{definition}[Homologically $\ell$-equivalent]
\label{def:ell-equivalent}
We will say that  \emph{$X$ is homologically $\ell$-equivalent to $Y$}  (denoted $X \sim_\ell Y$), if and only if there exists 
spaces, $X=X_0,X_1,\ldots,X_n=Y$ and homological $\ell$-equivalences  $f_1,\ldots,f_{n}$ as shown below:
\[
\xymatrix{
&X_1 \ar[ld]_{f_1}\ar[rd]^{f_2} &&X_3\ar[ld]_{f_3} \ar[rd]^{f_4}& \cdots&\cdots&X_{n-1}\ar[ld]_{f_{n-1}}\ar[rd]^{f_{n}} & \\
X_0 &&X_2  && \cdots&\cdots &&  X_n&
}.
\]
It is clear that $\sim_\ell$ is an equivalence relation.
\end{definition}

\begin{definition}[Diagrams of topological spaces]
\label{def:diagram-of-spaces}
A \emph{diagram of topological spaces} is a functor, $X:J \rightarrow \Top$, from a small category $J$ to $\Top$. 
\end{definition}

We extend Definition~\ref{def:equivalence-spaces} to diagrams of topological spaces.
We denote by $\Top$ the category of topological spaces.

\begin{definition}[Homological $\ell$-equivalence between diagrams of topological spaces]
\label{def:equivalence-diagrams}
Let $J$ be a small category, and $X,Y: J \rightarrow \Top$ be two functors. 

We will say that  \emph{a diagram $X:J \rightarrow \Top$ is homologically $\ell$-equivalent to the diagram $Y:J \rightarrow \Top$}  (denoted as before by $X \sim_\ell Y$), if and only if there exists diagrams
 $X=X_0,X_1,\ldots,X_n=Y:J \rightarrow \Top$ and homological $\ell$-equivalences  $f_1,\ldots,f_{n}$ as shown below:
\[
\xymatrix{
&X_1 \ar[ld]_{f_1}\ar[rd]^{f_2} &&X_3\ar[ld]_{f_3} \ar[rd]^{f_4}& \cdots&\cdots&X_{n-1}\ar[ld]_{f_{n-1}}\ar[rd]^{f_{n}} & \\
X_0 &&X_2  && \cdots&\cdots &&  X_n&
}.
\]
It is clear that $\sim_\ell$ is an equivalence relation.
\end{definition}

\begin{notation} [Diagram of various unions of a finite number of subspaces]
\label{not:diagram-Delta}
Let $J$ be a finite set, $A$ a topological space, 
and $\mathcal{A} = (A_j)_{j \in J}$ a tuple of subspaces of $A$  indexed by $J$.

For any subset 
$J' \subset J$, we denote 
\begin{eqnarray*}
\mathcal{A}^{J'} &=& \bigcup_{j' \in J'} A_{j'}.
\end{eqnarray*}

We consider $2^J$ as a category whose objects are elements of $2^J$, and whose only morphisms 
are given by: 
\begin{eqnarray*}
2^J(J',J'') &=& \emptyset  \mbox{ if  } J' \not\subset J'', \\
2^J(J',J'') &=& \{\iota_{J',J''}\} \mbox{  if } J' \subset J''.
\end{eqnarray*} 
We denote by $\Simp^J(\mathcal{A}):2^J \rightarrow \Top$ the functor (or the diagram) defined by
\[
\Simp^J(\mathcal{A})(J') = \mathcal{A}^{J'}, J' \in 2^J,
\]
and
$\Simp^J(\mathcal{A})(\iota_{J',J''})$ is the inclusion map $\mathcal{A}^{J'} \hookrightarrow \mathcal{A}^{J''}$.
\end{notation}

\begin{notation}
For $N \in \Z$ we denote by $[n] = \{0,\ldots,N\}$. In particular, $[-1] = \emptyset$. 
\end{notation}

Armed with the definition of homological equivalence of diagrams as defined above
we are finally in a position to state the specifications of the simplicial replacement
algorithm described in \cite{Basu-Karisani}.

\begin{algorithm}[H]
\caption{(Simplicial replacement)}
\label{alg:simplicial-replacement}
\begin{algorithmic}[1]
\INPUT
\Statex{
\begin{enumerate}[1.]
\item
A finite set of polynomials $\mathcal{P} \subset \D[X_1,\ldots,X_k]$;
\item
an integer $N \geq 0$, and for each $i \in [N]$, a $\mathcal{P}$-closed formula $\phi_i$;
\item 
$\ell, 0 \leq \ell \leq k$.
\end{enumerate}
 }
 \OUTPUT
 \Statex{
A simplicial complex $\Delta$ and for each $I \subset [N]$ a subcomplex $\Delta_I \subset \Delta$ such that there is a diagrammatic homological $\ell$-equivalence
\[
(I \mapsto \Delta_I)_{I \subset [N]} \overset{h}{\sim}_\ell \Simp^{[N]}(\RR(\Phi)),
\]
where $\Phi(i) = \phi_i, i \in [N]$.
}
\COMPLEXITY
The complexity of the algorithm is 
bounded by $(sd)^{k^{O(\ell)}}$, where $s =\card(\mathcal{P})$ and $d = \max_{P \in \mathcal{P}} \deg(P)$.
\end{algorithmic}
\end{algorithm}

\begin{notation}
\label{not:graph}
A finite  (directed)  graph $G$ is a tuple $(V(G), E(G), \mathrm{head}, \mathrm{tail})$,
where $V(G), E(G)$ are finite sets and $\mathrm{head}, \mathrm{tail}: E(G) \rightarrow V(G)$ are maps. To every finite graph $G$ there is an one-dimensional regular cell complex 
associated naturally to it. We will denote this cell complex by $|G|$.
\end{notation}

\subsubsection{Algorithm for computing basis of $\HH_1(S)$}
We are now in a position to describe the algorithm that will accomplish
Step~\ref{itemlabel:prob-H1:2} of our algorithm for computing a semi-algebraic basis
of the first homology of a given closed semi-algebraic set.

The outline of the algorithm is as follows. We first replace the given closed semi-algebraic set by one that is closed and bounded and which is homologically
equivalent to the given set. This is accomplished in Lines~
\ref{alg:basis:line:1}, \ref{alg:basis:line:2} and \ref{alg:basis:line:3} using Algorithm~\ref{alg:big-enough-radius} (Big enough radius).
We then use Algorithm~\ref{alg:surjection} (Computing one-dimensional subset) 
to compute a one-dimensional subset
$\Gamma$ (described in terms of curve-segment representations and real univariate representations) having the property that $\HH_q(S,\Gamma) = 0$ for $q=0,1$, where $S$ is the given semi-algebraic set. This is accomplished in Line~\ref{alg:basis:line:4}.
We then use Algorithm~\ref{alg:convert-to-closed} (Conversion of curve segment representations to closed formulas)  in Line~\ref{alg:basis:line:5}
to convert the description of $\Gamma$ obtained in the last step into a formula.
In Line~\ref{alg:basis:line:6},
we extract the underlying structure of a combinatorial graph $G$ from $\Gamma$ and 
compute a set of simple cycles which spans the cycle space of $G$.
Finally in Line~\ref{alg:basis:line:7}, we use Algorithm~\ref{alg:simplicial-replacement} (Simplicial replacement)  and Gauss-Jordan elimination to compute 
a minimal subset of the simple cycles computed in the previous step which span $\HH_1(S)$, and these form a semi-algebraic basis of $\HH_1(S)$.

\begin{algorithm}[H]
\caption{(Computing homology basis)}
\label{alg:basis}
\begin{algorithmic}[1]
\INPUT
\Statex{
\begin{enumerate}[1.]
\item
a finite set $\mathcal{P} \subset \D[X_1,\ldots,X_k]$;
\item
a $\mathcal{P}$-closed formula $\Phi$.
\end{enumerate}
}

\OUTPUT
 \Statex{
 \begin{enumerate}[1.]
    \item
    a finite set $\mathcal{Q} \subset \D[X_1,\ldots,X_k]$;
    \item a finite tuple $(\Psi_j)_{j \in J}$, in which each $\Psi_j$ is a $\mathcal{Q}$-formula, such that
the realizations $\Gamma_j = \RR(\Psi_j,\R^k)$ have the following properties:
\begin{enumerate}[(a)]
\item For each $j \in J$, 
$\Gamma_j \subset S$ and  $\Gamma_j$ is semi-algebraically homeomorphic to 
$\Sphere^1$;
\item the inclusion map $\Gamma_j \hookrightarrow S$ induces an injective map
$\FF \cong \HH_1(\Gamma_j) \rightarrow \HH_1(S)$, whose image we denote by $[\Gamma_j]$;
\item
the tuple $([\Gamma_j])_{j \in J}$ forms a basis of $\HH_1(S)$.
 \end{enumerate}
 \end{enumerate}
 }

\PROCEDURE
\State{Use Algorithm~\ref{alg:big-enough-radius} (Big enough radius) with input $(\mathcal{P},\Phi)$ and let
$r = \frac{a}{b} > 0, a,b\in \D$ be the output}.
\label{alg:basis:line:1}
\State{$\mathcal{P} \gets \mathcal{P} \cup \{b^2(X_1^2 + \cdots + X_k^2) - a^2\}$.}
\label{alg:basis:line:2}
\State{$\Phi \gets \Phi \wedge (b^2(X_1^2 + \cdots + X_k^2) - a^2 \leq 0)$.}
\label{alg:basis:line:3}

\State{Call Algorithm~\ref{alg:surjection} (Computing one-dimensional subset)  with input $(\mathcal{P},\Phi)$ to obtain
\begin{enumerate}[(a)]
    \item a finite set $\mathcal{U}$ of real univariate representations over $\mathcal{T}$;
    \item a finite indexing set $I$ and
     a finite tuple $(\gamma_j)_{j \in I}$ where each $\gamma_j$  a curve segment representation over $\mathcal{T}$;
    \item mappings $L,R:I \rightarrow \mathcal{U}$, defined by $\ass(L(j))$ is the left endpoint of $\gamma_j$, and $\ass(R(j))$ is the right endpoint of $\gamma_j$.
 \end{enumerate}
} 
\label{alg:basis:line:4}

\State{Using Algorithm~\ref{alg:convert-to-closed} (Conversion of curve segment representations to closed formulas)  compute for each $j \in I$
a set of polynomials $\mathcal{Q}_j$ and a 
$\mathcal{Q}_j$-closed
formula
$\Theta_j$ such that
$\RR(\Theta_j) = \clos(\ass(\gamma_j))$.}
\label{alg:basis:line:5}
\State{$\mathcal{Q} \gets \bigcup_{j \in I} \mathcal{Q}_j$.}

\State{$G \gets (E=I,V = \mathcal{U},\head= L, \tail = R)$.}
\State{Using a graph traversal algorithm compute a tuple $(C_1,\ldots,C_N)$ where each 
$C_h= (i_{h,0},\ldots,i_{h,q_h-1}) \in I^{q_h}$ and  
\begin{enumerate}[(a)]
    \item $\tail(i_{j,h}) = \head(i_{j,h+1 \mod  q_j}), h = 0,\ldots,q_j -1$.
    \item $C_1, \ldots,C_N$ are simple cycles of $G$.
    \item The cycles $C_1, \ldots, C_N$ form a basis of the cycle space of $G$ (which is isomorphic to $\HH_1(|G|)$).
\end{enumerate}
}
\label{alg:basis:line:6}

\algstore{myalg}
\end{algorithmic}
\end{algorithm}
 
\begin{algorithm}[H]
\begin{algorithmic}[1]
\algrestore{myalg}

\For {$1 \leq h \leq N$} 
    \State{$\Psi_h \gets \Theta_{i_{h,0}} \vee \cdots \vee \Theta_{i_{h,q_h-1}}$.}    
\EndFor
\State{Call Algorithm~\ref{alg:simplicial-replacement} (Simplicial replacement)  with input:
\begin{enumerate}[1.]
\item
$\mathcal{Q}$,
\item the tuple of $\mathcal{Q}$-closed formulas
$\boldsymbol{\Phi} = (\phi_0,\ldots,\phi_N) = (\Psi_1,\ldots,\Psi_N,\Phi)$,
\item
$\ell =1$
\end{enumerate}
to obtain
a simplicial complex $\Delta_{1}(\boldsymbol{\Phi})$ such that 
  \[
  (J \mapsto |\Delta_{1}(\boldsymbol{\Phi}|_{J})|)_{J \subset [N]} 
  \]
  is homologically $1$-equivalent (cf. Definition~\ref{def:ell-equivalent})  to 
  $\Simp^{[N]}(\RR(\boldsymbol{\Phi}))$ (cf. Notation~\ref{not:diagram-Delta}).
}

\State{
Using Gauss-Jordan elimination identify a minimal subset  $J \subset \{1,\ldots,N \}$
such that the $\mathrm{span}(\{ \mathrm{Im}(\HH_1(\Delta_{1}(\boldsymbol{\Phi}|_{\{h\}})) \rightarrow 
\HH_1(\Delta_{1}(\boldsymbol{\Phi}))) \;\vert\; h \in J\}) = \HH_1(\Delta_{1}(\boldsymbol{\Phi}))$.
}
\label{alg:basis:line:7}
\State{Output $(\Psi_h)_{h \in J}$.}

\COMPLEXITY{ The complexity of each $\Psi_h, h \in J$ is bounded by $(s d)^{O(k^2)}$, and the complexity of the algorithm  is bounded by $(s d)^{k^{O(1)}}$, where $s = \card(\mathcal{P})$ and $d = \max_{P \in \mathcal{P}} \deg(P)$.}
\end{algorithmic}
\end{algorithm}

\begin{proof}[Proof of correctness of Algorithm~\ref{alg:basis}]
The correctness of Algorithm~\ref{alg:basis} follows from the correctness
of Algorithms~\ref{alg:big-enough-radius}, \ref{alg:modified-curve-segment}, \ref{alg:convert-to-closed}, and \ref{alg:simplicial-replacement}.
\end{proof}

\begin{proof}[Complexity analysis of Algorithm~\ref{alg:basis}]
The complexity upper bound is a consequence  of the complexity analysis of the  Algorithms~\ref{alg:big-enough-radius}, \ref{alg:modified-curve-segment}, \ref{alg:convert-to-closed}, and \ref{alg:simplicial-replacement}.
\end{proof}

\begin{proof}[Proof of Theorem~\ref{thm:main}]
Theorem~\ref{thm:main} follows from the correctness and complexity analysis of 
Algorithm~\ref{alg:basis}.
\end{proof}

\begin{proof}[Proof of Theorem~\ref{thm:basic}]
Theorem~\ref{thm:basic} follows from the correctness and complexity analysis of 
Algorithms~\ref{alg:surjection} (Computing one-dimensional subset) and \ref{alg:convert-to-closed} (Conversion of curve segment representations to closed formulas).
\end{proof}

\section{Conclusion and open problems}
In this paper we have proved the existence of an algorithm with singly exponential complexity for obtaining a semi-algebraic basis of the first homology group of 
of a given closed semi-algebraic set (generalizing existing algorithm in the case of  the zero-th
dimensional homology). One obvious open problem is to generalize this result to
the higher homology groups. As an intermediate problem we can ask for the solution
to Conjecture~\ref{conj:basic}. The techniques developed in the current paper
may possibly generalize to prove Conjecture~\ref{conj:basic}, but there are 
formidable technical problems to overcome and we leave this to future work.

\bibliographystyle{amsplain}
\bibliography{master}
 
\end{document}